\documentclass[11pt]{article}
\usepackage{latexsym,amsfonts,amssymb,amsmath,amsthm}

\usepackage{color,comment}

\begin{document}

\newtheorem{tm}{Theorem}[section]
\newtheorem{rk}{Remark}[section]
\newtheorem{prop}{Proposition}[section]
\newtheorem{defin}{Definition}[section]
\newtheorem{coro}{Corollary}[section]

\newtheorem{lem}{Lemma}[section]
\newtheorem{assumption}{Assumption}[section]

\newtheorem{nota}[tm]{Notation}
\numberwithin{equation}{section}

\newcommand{\stk}[2]{\stackrel{#1}{#2}}
\newcommand{\dwn}[1]{{\scriptstyle #1}\downarrow}
\newcommand{\upa}[1]{{\scriptstyle #1}\uparrow}
\newcommand{\nea}[1]{{\scriptstyle #1}\nearrow}
\newcommand{\sea}[1]{\searrow {\scriptstyle #1}}
\newcommand{\csti}[3]{(#1+1) (#2)^{1/ (#1+1)} (#1)^{- #1
 / (#1+1)} (#3)^{ #1 / (#1 +1)}}
\newcommand{\RR}[1]{\mathbb{#1}}

\newcommand{ \bl}{\color{blue}}
\newcommand {\rd}{\color{red}}
\newcommand{ \bk}{\color{black}}
\newcommand{ \gr}{\color{OliveGreen}}
\newcommand{ \mg}{\color{RedViolet}}

\newcommand{\ep}{\varepsilon}
\newcommand{\rr}{{\mathbb R}}
\newcommand{\alert}[1]{\fbox{#1}}

\newcommand{\eqd}{\sim}
\newcommand\PP{\ensuremath{\mathbb{P}}}
\def\R{{\mathbb R}}
\def\N{{\mathbb N}}
\def\Q{{\mathbb Q}}
\def\C{{\mathbb C}}
\def\l{{\langle}}
\def\r{\rangle}
\def\t{\tau}
\def\k{\kappa}
\def\a{\alpha}
\def\la{\lambda}
\def\De{\Delta}
\def\de{\delta}
\def\ga{\gamma}
\def\Ga{\Gamma}
\def\ep{\varepsilon}
\def\eps{\varepsilon}
\def\si{\sigma}
\def\Re {{\rm Re}\,}
\def\Im {{\rm Im}\,}
\def\E{{\mathbb E}}
\def\P{{\mathbb P}}
\def\Z{{\mathbb Z}}
\def\D{{\mathbb D}}
\newcommand{\ceil}[1]{\lceil{#1}\rceil}

\title{Long time behavior of random and nonautonomous  Fisher-KPP equations. Part II. Transition fronts}

\author{
Rachidi B. Salako and Wenxian Shen\thanks{Partially supported by the NSF grant DMS--1645673}\\
Department of Mathematics and Statistics\\
Auburn University\\
Auburn University, AL 36849\\
U.S.A. }

\date{}
\maketitle
\begin{abstract}
\noindent  In the current series of two papers, we  study the long time behavior of  the following random Fisher-KPP equation
$$
u_t =u_{xx}+a(\theta_t\omega)u(1-u),\quad x\in\R,
\eqno(1)
$$
where $\omega\in\Omega$, $(\Omega, \mathcal{F},\mathbb{P})$ is a given probability space, $\theta_t$ is  an ergodic metric dynamical system on $\Omega$,  and $a(\omega)>0$ for every $\omega\in\Omega$.
We also study the long time behavior of   the following nonautonomous Fisher-KPP equation,
$$
u_t=u_{xx}+a_0(t)u(1-u),\quad x\in\R,
\eqno(2)
$$
where $a_0(t)$ is a positive locally H\"older continuous function.
In the first part of the series, we studied the stability of positive equilibria and the spreading speeds of (1) and (2).
In this second part of the series, we investigate the existence and stability of transition fronts of (1) and (2).
We first study the transition fronts of (1).
 Under some proper assumption on $a(\omega)$, we show the existence of random transition fronts of (1)  with least mean speed greater than or equal to some constant $\underline{c}^*$ and the nonexistence of ranndom transition fronts of (1) with least mean speed less than $\underline{c}^*$. We  prove the stability of random transition fronts of (1) with least mean speed greater
than $\underline{c}^*$. Note that  it is proved in the first part  that $\underline{c}^*$ is  the infimum of the spreading speeds of (1).
We next study the existence and stability of transition fronts of (2).
It is not assumed that $a(\omega)$ and $a_0(t)$ are bounded above and
below by some positive constants. Many existing results in literature on transition fronts of Fisher-KPP equations have
been extended to the general cases considered in the current paper. The current paper also obtains several new results.
\end{abstract}



\section{Introduction and statements of the main results}

The current series of two papers is concerned with the long time behavior   of the following random Fisher-KPP equation,
\begin{equation}\label{Main-eq}
u_t =u_{xx}+a(\theta_t\omega)u(1-u),\quad x\in\R,
\end{equation}
where $\omega\in\Omega$,  $(\Omega, \mathcal{F},\PP,\{\theta_t\}_{t\in \R})$ is an ergodic  metric dynamical system  on $\Omega$,
 $a:\Omega\to (0,\infty)$ is measurable, and $a^\omega(t):=a(\theta_t\omega)$ is locally H\"older continuous for every $\omega\in\Omega$.
It also considers the long time behavior   of  the following nonautonomous Fisher-KPP equation,
\begin{equation}\label{nonautonomous-eq}
u_t=u_{xx}+a_0(t)u(1-u), x\in\R,
\end{equation}
where $a_0:\R\to (0,\infty)$ is locally H\"older continuous.

Observe that \eqref{Main-eq} (resp. \eqref{nonautonomous-eq}) with $a(\omega)\equiv 1$ (resp. with $a_0(t)\equiv 1$) becomes
\begin{equation}
\label{fisher-kpp}
u_t=u_{xx}+u(1-u),\quad x\in\R.
\end{equation}
Equation \eqref{fisher-kpp}  is called in literature  Fisher-KPP   equation
 due to the pioneering papers of Fisher
\cite{Fisher} and Kolmogorov, Petrowsky, Piskunov \cite{KPP} on traveling wave solutions and take-over properties of \eqref{fisher-kpp}.
It is clear that the constant solution $u=1$ of \eqref{fisher-kpp} is asymptotically stable with respect to strictly positive perturbations.
Fisher in
\cite{Fisher} found traveling wave solutions $u(t,x)=\phi(x-ct)$ of \eqref{fisher-kpp}
$(\phi(-\infty)=1,\phi(\infty)=0)$ of all speeds $c\geq 2$ and
showed that there are no such traveling wave solutions of slower
speed. He conjectured that the take-over occurs at the asymptotic
speed $2$. This conjecture was proved in \cite{KPP}  {for some special initial distribution and was proved in \cite{ArWe2} for general initial distributions.
 More precisely, it is proved
in \cite{KPP} that for the  nonnegative solution $u(t,x)$ of \eqref{fisher-kpp} with
$u(0,x)=1$ for $x<0$ and $u(0,x)=0$ for $x>0$, $\lim_{t\to \infty}u(t,ct)$ is $0$ if $c>2$ and $1$ if $c<2$. It is proved
in \cite{ArWe2} that for any
nonnegative solution $u(t,x)$ of (\ref{fisher-kpp}), if at
time $t=0$, $u$ is $1$ near $-\infty$ and $0$ near $\infty$, then
$\lim_{t\to \infty}u(t,ct)$ is $0$ if $c>2$ and $1$ if $c<2$.}
In
literature, $c^*=2$ is   called the {\it
spreading speed} for \eqref{fisher-kpp}.

A huge amount of research has been carried out toward various extensions of
 traveling wave solutions and take-over properties  of \eqref{fisher-kpp} to general time and space independent
as well as time and/or space dependent Fisher-KPP type  equations.
See, for example,  \cite{ArWe1, ArWe2, Bra, Ham, Kam, Sat, Uch}, etc., for
the extension to general time and space independent Fisher-KPP type equations;  see
 \cite{BeHaNa1,  BeHaRo,  FrGa, HaRo, HaRos,  HuZi1,
LiYiZh, LiZh, LiZh1,  Nad,  NoRuXi, NoXi,  Wei1, Wei2},  and references therein for
the extension to time and/or space periodic Fisher-KPP
type equations; and see
\cite{BeHaRoq, BeHaRos,BeHaNa2, BeHa07, BeHa12,  HePaSt, HuSh,  Mat, Nad1,  NaRo1, NaRo2, NaRo3, NoRoRyZl,  She6, She7, She8,  She9, TaZhZl, Xin1, Zla}, and references therein for
the extension to quite general time and/or space dependent Fisher-KPP
type equations. The reader is referred to \cite{FaHuWu, HeWu, ZoWu}, etc.
   for the study of Fisher-KPP reaction diffusion equations with time delay.

 All the existing works on \eqref{Main-eq} (resp. \eqref{nonautonomous-eq})  assumed $\inf_{t\in\R} a^\omega(t)>0$ and $a^\omega(\cdot)\in L^\infty(\R)$ (resp. $\inf_{t\in\R} a_0(t)>0$ and $\sup_{t\in\R} a_0(t)<\infty$).
The objective of the current series of two papers  is to study the long time behavior, in particular, the stability of positive constant solutions, the spreading speeds, and the transition fronts  of \eqref{Main-eq} (resp. \eqref{nonautonomous-eq}) without the assumption $\inf_{t\in\R} a^\omega(t)>0$ and $a^\omega(\cdot)\in L^\infty(\R)$ (resp. without the assumption $\inf_{t\in\R} a_0(t)>0$ and $\sup_{t\in\R} a_0(t)<\infty$). The applications of the results established for \eqref{Main-eq} to Fisher-KPP equations whose growth rate and/or carrying
 capacity are perturbed by real noises will  also be discussed.

In the first part of the series, we studied the stability of positive constant solutions and the spreading speeds of \eqref{Main-eq} and
\eqref{nonautonomous-eq} (see section 2 for the review of some results established in the first part).
In this second part of the series, we investigate the existence and stability of transitions fronts of \eqref{Main-eq} and
\eqref{nonautonomous-eq}.
Note that the so called periodic traveling wave solutions or pulsating traveling fronts to time and/or space periodic
reaction diffusion equations are
 natural extension of the notion of traveling wave solutions in the classical sense,  and that  the so called transition fronts or generalized traveling waves
 to general time and/or space dependent reaction equations are the natural extension of the notion  traveling wave solutions in the classical sense
 (see  \cite{BeHa07, BeHa12} for the introduction of the notion of transition fronts or generalzed traveling waves  in the general case, and
  \cite{Mat, She4, She7, She8} for the time almost periodic or space almost periodic cases).

 It should be pointed out that the work \cite{NaRo1} studied the existence of transition fronts of \eqref{Main-eq} with mean speed
 greater than  some number $c^*$ under the assumption
that $\inf_{t\in\R} a^\omega(t)>0$ and $a^\omega(\cdot)\in L^\infty(\R)$. The work \cite{NaRo1} also studied the existence of transition fronts of \eqref{nonautonomous-eq}
 with least  mean speed
 greater than  some number $c^*$  under the assumption
that $\inf_{t\in\R} a_0(t)>0$ and $\sup_{t\in\R} a_0(t)<\infty$. Among others, the work \cite{She9} considered the stability of
transition fronts of \eqref{nonautonomous-eq}
 under the assumption
that $\inf_{t\in\R} a_0(t)>0$ and $\sup_{t\in\R} a_0(t)<\infty$.

The objective of the current paper is to study the existence and stability of transition fronts of \eqref{Main-eq}  without the assumption $\inf_{t\in\R} a^\omega(t)>0$ and $a^\omega(\cdot)\in L^\infty(\R)$, and to study the existence and stability
 of transition fronts of \eqref{nonautonomous-eq}  without the assumption $\inf_{t\in\R} a_0(t)>0$ and $\sup_{t\in\R} a_0(t)<\infty$.
 Most results in \cite{NaRo1} and \cite{She9} are extended to such general cases and some new results are  obtained in this paper.

We first consider \eqref{Main-eq}. As in the first part, we introduce the following notations and assumption related to \eqref{Main-eq}.  Let
\begin{equation}\label{a-least-mean}
\underline{a}(\omega)=\liminf_{t-s\to\infty}
\frac{1}{t-s}\int_s^ta(\theta_{\tau}\omega)d\tau:=\lim_{r\to\infty}\inf_{t-s\ge r}\frac{1}{t-s}\int_s^ta(\theta_{\tau}\omega)d\tau
\end{equation}
and
\begin{equation}
\label{a-largest-mean}
\overline{a}(\omega)=\limsup_{t-s\to\infty}\frac{1}{t-s}\int_s^ta(\theta_\tau\omega)d\tau:=\lim_{r\to\infty}\sup_{t-s\ge r}\frac{1}{t-s}\int_s^ta(\theta_{\tau}\omega)d\tau.
\end{equation}
 Observe that
\begin{equation}
\label{a-a-eq1}
\underline{a}(\theta_t\omega)=\underline{a}(\omega)\quad {\rm and}\quad \overline{a}(\theta_t\omega)=\overline{a}(\omega),\,\, \forall\,\, t\in\R,
\end{equation}
and that
$$
\underline{a}(\omega)=\liminf_{t,s\in\Q,t-s\to\infty}\frac{1}{t-s}\int_s^t a(\theta_\tau)d\tau\,\, \,\, {\rm and}\,\, \,\, \overline{a}(\omega)=\liminf_{t,s\in\Q,t-s\to\infty}\frac{1}{t-s}\int_s^t a(\theta_\tau)d\tau.
$$
Then by the countability of the set $\Q$ of rational numbers, both $\underline{a}(\omega)$ and $\overline{a}(\omega)$ are measurable in $\omega$.

Throughout this paper, we assume that  the following standing assumption holds.
\medskip

\noindent {\bf (H1)}
{\it  $0< \underline{a}(\omega)\le \overline{a}(\omega)<\infty$ for a.e. $\omega\in\Omega$.}

\medskip
Note that
  {\bf (H1)} implies that $\underline{a}(\cdot),a(\cdot),\hat a(\cdot)\in  L^1 (\Omega, \mathcal{F},\PP)$ (see Lemma \ref{prelim-lm1}).
Assume {\bf (H1)}. Then by the ergodicity of the metric dynamical system $(\Omega, \mathcal{F},\PP,\{\theta_t\}_{t\in \R})$, there are
 $\hat a, \underline{a}, \bar a\in\R^+$ and a measurable subset $\Omega_0\subset\Omega$ with $\P(\Omega_0)=1$ such that
 \begin{equation}
 \label{omega-0}
 \begin{cases}
 \theta_t\Omega_0=\Omega_0\quad \forall\,\, t\in\R\cr
 \lim_{t\to \pm \infty}\frac{1}{t}\int_0^t a(\theta_\tau\omega)d\tau=\hat a\quad \forall\,\, \omega\in\Omega_0\cr
 \liminf_{t-s\to\infty}\frac{1}{t-s}\int_s^t a(\theta_\tau\omega)d\tau =\underbar a\quad \forall\,\,\omega\in\Omega_0\cr
 \limsup_{t-s\to\infty}\frac{1}{t-s}\int_s^t a(\theta_\tau)d\tau=\bar a\quad\forall\,\, \omega\in\Omega_0.
 \end{cases}
 \end{equation}

In the following, we roughly state the main results of the current  paper.
For given $u_0\in X:=C_{\rm unif}^b(\R)$ and $\omega\in\Omega$, let $u(t,x;u_0,\omega)$ be the solution of \eqref{Main-eq} with $u(0,x;u_0,\omega)=u_0(x)$.
Note that, for $u_0\in X$ with $u_0\ge 0$, $u(t,x;u_0,\omega)$  exists  for $t\in [0,\infty)$ and $u(t,x;u_0,\omega)\ge 0$ for all $t\ge 0$.
Note also that $u\equiv 0$ and $u\equiv 1$ are two constant solutions of \eqref{Main-eq}.

A solution $u(t,x;\omega)$ of \eqref{Main-eq} is called
an {\it entire solution} if it is a solution on $t\in(-\infty,\infty)$.
An entire solution $u(t,x;\omega)$ is called a {\it random  traveling wave solution} or a {\it random transition front} of \eqref{Main-eq} connecting  $1$ and $0$ if for a.e. $\omega\in\Omega$,
\begin{equation}
\label{front-eq}
u(t,x;\omega)=U(x-C(t;\omega),\theta_t\omega)
 \end{equation}
 for some $U(\cdot,\omega)$ and $C(\cdot;\omega)$,  where $U(x,\omega)$ and $C(t;\omega)$ are measurable in $\omega$, and for a.e.
 $\omega\in\Omega$,
\begin{equation}
\label{tail-eq}
0<U(x,\omega)<1,\,\, {\rm and}\,\,
\lim_{x\to -\infty}U(x,\theta_t\omega)=1,\,\, \lim_{x\to\infty}U(x,\theta_t\omega)=0\,\,\, \text{uniformly in}\,\,  t\in\R.
\end{equation}
We may write $\lim_{x\to -\infty}U(x,\omega)=1$ and $\lim_{x\to\infty}U(x,\omega)=0$ (if the limits exist)  as $U(-\infty,\omega)=1$ and $U(\infty,\omega)=0$, respectively.
Suppose that $u(t,x;\omega)=U(x-C(t;\omega);\theta_t\omega)$ is a random transition front of \eqref{Main-eq}.
 If $U_x(x,\omega)<0$ for a.e. $\omega\in\Omega$ and all $x\in\R$, $u(t,x;\omega)=U(x-C(t;\omega);\theta_t\omega)$ is said to be
a {\it monotone random transition front}.  If there is $\underbar c\in\R$ such that  for a.e. $\omega\in\Omega$,
\begin{equation}
\label{least-mean-speed-eq}
\liminf_{t-s\to\infty}\frac{C(t;\omega)-C(s;\omega)}{t-s}=\underbar c,
\end{equation}
then $\underline{c}$
is called its {\it least mean speed} or {\it least average speed}. If there is $\hat c\in\R$ such that for a.e. $\omega\in\Omega$,
\begin{equation}
\label{mean-speed-eq}
{ \lim_{t\to\infty}\frac{C(t;\omega)}{t}=\hat c,}
\end{equation}
then $\hat c$ is called its {\it mean speed} or {\it average speed}.

 Recall that when $a(\omega)\equiv a$, \eqref{Main-eq} becomes classical Fisher-KPP equation, and  that the classical Fisher-KPP equation has a unique (up to phase translation) traveling wave solution $u(t,x)=\phi(x-ct)$ connecting $u=0$ and $u=1$,
(i.e. $\phi(-\infty)=1$ and $\phi(\infty)=0$) with speed $c\ge 2 \sqrt a$ and has no such traveling wave solution with speed $c<2\sqrt{a}$.






 Let
\begin{equation}
\label{minimal-speeds-eq}
\underline{c}^*=2\sqrt {\underline{a}},\quad  \hat c^*=2\sqrt{\hat{a}},\quad   {\rm and}\quad \overline{c}^*=2\sqrt {\overline{a}}.
\end{equation}
 We have the following results on the existence and stability of random  transition fronts  of \eqref{Main-eq} connecting $1$ and $0$.

\medskip

\noindent {\bf i)}  (Existence of random transition fronts)  {\it For any given $c > \underline{c}^*$, there is a  monotone
random transition front of \eqref{Main-eq} with least mean speed $\underline{c}=c$}  (see Theorem  \ref{Existence of random transition front}(1)).

\smallskip

\noindent {\bf ii)} (Existence of critical random transition fronts)  {\it There is a
 monotone random transition front of \eqref{Main-eq}  with least mean speed $\underline{c}=\underline{c}^*$}  (see Theorem  \ref{Existence of random transition front}(2)).

\medskip

\noindent {\bf iii)} (Nonexistence of random transition fronts) {\it There is no random transition front of \eqref{Main-eq}  with  least   mean speed  less than  $\underline{c}^*$}  (see Theorem  \ref{Existence of random transition front}(3)).

\medskip

\noindent {\bf iv)} (Stability of random transition fronts) {\it The random transition front established in ii)  is  asymptotically stable}   (see  Theorem \ref{Stability tm}).

\medskip

\noindent {\bf v)} (Average speed of random transition fronts) {\it Suppose that $u(t,x;\omega)=U(x-C(t;\omega);\theta_t\omega)$ is a
 monotone random transition front of \eqref{Main-eq}. Then its  average speed { $\hat c=\lim_{t\to\infty}\frac{C(t;\omega)}{t}$}  exists and $\hat c\ge 2\sqrt{\hat a}$} (see Theorem \ref{average-speed-thm}).

\medskip


\medskip

Next, we consider \eqref{nonautonomous-eq}. A solution $u(t,x)$ of \eqref{nonautonomous-eq} is called
an {\it entire solution} if it is a solution on $t\in(-\infty,\infty)$.
An entire solution $u(t,x)$ of  \eqref{nonautonomous-eq} is called  a {\it transition front} of \eqref{nonautonomous-eq} connecting  $1$ and $0$ if
\begin{equation}
\label{front--nonauton-eq}
u(t,x)=U(x-C(t),t)
 \end{equation}
 for some $U(x,t)$ and $C(t)$,  where $U(x,t)$ satisfies
\begin{equation}
\label{tail-nonauton-eq}
0<U(x,t)<1,\quad
U(-\infty,t)=1\quad {\rm and}\quad U(\infty,t)=0\quad  \text{uniformly in}\,\,  t\in\R.
\end{equation}
Suppose that $u(t,x)=U(x-C(t),t)$ is a transition front of \eqref{nonautonomous-eq}.  Then
\begin{equation}
\label{least-mean-speed-nonauton-eq}
\underline{c}=\liminf_{t-s\to\infty}\frac{C(t)-C(s)}{t-s}
\end{equation}
is called its {\it least  mean speed}.  Let {\bf (H2)} be the following standing assumption.

\medskip

\noindent {\bf (H2)}  {\it  $0< \underline{a}_0:= \liminf_{t-s\to\infty}\frac{1}{t-s}\int_s^t a_0(\tau)d\tau\le \overline{a}_0:=\limsup_{t-s\to\infty}\frac{1}{t-s}\int_s^ta_0(\tau)d\tau<\infty$.}

\medskip

The assumption {\bf (H2)} is the analogue of {\bf (H1)}. We will give some example for $a_0(\cdot)$ in section 6.

Assume  {\bf (H2)}.  Let
\begin{equation}
\bar c_0^*=2\sqrt{\bar{a}_0}\quad {\rm and}\quad \underline{c}_0^*=2\sqrt {\underline{a}_0}.
\end{equation}
Similar to i)-iv) for \eqref{Main-eq},  we have the following results on the existence and stability of   transition fronts  of \eqref{nonautonomous-eq} connecting $1$ and $0$.

\medskip

\noindent {\bf i$)'$}  (Existence of  transition fronts)  {\it For any given  $c>\underline{c}_0^*$, there is a
 transition front of \eqref{nonautonomous-eq}  with least mean speed $\underline{c}=c$  (see Theorem  \ref{nonautonomous-thm1}(1)).}

\smallskip

\noindent {\bf ii$)'$} (Existence of critical  transition fronts)  {\it There is a
 transition front of \eqref{nonautonomous-eq}  with least mean speed $\underline{c}=\underline{c}_0^*$}  (see Theorem  \ref{nonautonomous-thm1}(2)).

\medskip

\noindent {\bf iii$)'$} \noindent { (Nonexistence of  transition fronts) {\it There is no  transition front of \eqref{nonautonomous-eq}  with  least   mean speed  less than  $\underline{c}_0^*$} (see Theorem  \ref{nonautonomous-thm1}(3)).

\medskip

\noindent {\bf iv$)'$} (Stability of transition fronts) {\it The  transition front established in i$)'$  is  asymptotically stable}   (see  Theorem \ref{nonautonomous-thm2}).

\medskip

We conclude the introduction with the following three remarks.

 First,  the results i$)'$ and  iii$)'$ extend  \cite[Theorem 2.3(1)]{NaRo1} and  \cite[Theorem 2.3(2)]{NaRo1} for \eqref{nonautonomous-eq}
with $0<\inf_{t\in\R}a_0(t)\le \sup_{t\in\R}a_0(t)<\infty$ to more general $a_0(t)$,  respectively,} and the result i) extends \cite[Theorem 2.5]{NaRo1} for \eqref{Main-eq} with
$0<\inf_{\omega\in\Omega}a(\omega)\le \sup_{\omega\in\Omega}a(\omega)<\infty$ to more general $a(\omega)$. The result iv$)'$ extends \cite[Theorem 2.2]{She9} for \eqref{nonautonomous-eq} with $0<\inf_{t\in\R}a_0(t)\le \sup_{t\in\R}a_0(t)<\infty$ to more general $a_0(t)$. The results ii), and ii$)'$ are new even for the cases that $0<\inf_{\omega\in\Omega}a(\omega)\le \sup_{\omega\in\Omega}a(\omega)<\infty$
and $0<\inf_{t\in\R}a_0(t)\le \sup_{t\in\R}a_0(t)<\infty$.

\smallskip

Second,  the results established for \eqref{Main-eq} and \eqref{nonautonomous-eq} can be applied to the following general random Fisher-KPP equation,
\begin{equation}
\label{general-random-eq}
u_t=u_{xx}+u(r(\theta_t\omega)-\beta(\theta_t\omega) u),
\end{equation}
where $r:\Omega\to (-\infty,\infty)$ and $\beta:\Omega\to (0,\infty)$ are measurable with locally H\"older continuous sample paths $r^\omega(t):=r(\theta_t\omega)$ and
$\beta^\omega(t):=\beta(\theta_t\omega)$,
and to the following general nonautonomous Fisher-KPP equation,
\begin{equation}
\label{general-nonautonomous-eq}
u_t=u_{xx}+u(r_0(t)-\beta_0(t) u),
\end{equation}
where $r_0:\R\to \R$ and $\beta_0:\R\to (0,\infty)$ are locally H\"older continuous. Note that \eqref{general-random-eq} models the population growth of a species with random
perturbations on its growth rate and carrying capacity, and \eqref{general-nonautonomous-eq} models the population growth of a species
with deterministic time dependent perturbations on its growth rate and carrying capacity.

In fact, under some assumptions on $r(\omega)$ and $\beta(\omega)$, it can be proved that
$$
u(t;\omega):=Y(\theta_t\omega)=\frac{1}{\int_{-\infty}^0 e^{-\int_s ^0 r(\theta_{\tau+t}\omega)d\tau}\beta(\theta_{s+t}\omega)ds}
$$
is a random equilibrium of \eqref{general-random-eq}. Let $\tilde u=\frac{u}{Y(\theta_t\omega)}$ and drop the tidle, \eqref{general-random-eq} becomes
\eqref{Main-eq} with $a(\theta_t\omega)=\beta(\theta_t\omega)\cdot Y(\theta_t\omega)$, and then the results established for \eqref{Main-eq}
can be applied to \eqref{general-random-eq}.
For example, consider the following random Fisher-KPP equation,
\begin{equation}
\label{real-noise-eq}
u_t=u_{xx}+u(1+ \xi(\theta_t\omega) -u),\quad x\in\R,
\end{equation}
 where $\omega\in\Omega$,  $(\Omega, \mathcal{F},\PP,\{\theta_t\}_{t\in \R})$ is an ergodic  metric dynamical system,  $\xi:\Omega\to \R$ is measurable,  and $\xi_t(\omega):=\xi(\theta_t\omega)$ is locally H\"older continuous
($\xi_t$ denotes a real noise or a colored noise).
Assume that $\xi_t(\cdot)$ satisfies the following {\bf (H3)}.

\medskip

\noindent  {\bf (H3)}  {\it $\xi:\Omega\to\R$ is measurable, $\int_\Omega |\xi(\omega)|d\PP(\omega)<\infty$,
$\int_\Omega \xi(\omega)d\PP(\omega)=0$,  $-1<\underline{\xi}(\omega)\le \overline{\xi}(\omega)<\infty$, { $\xi_{\inf}(\theta_{\cdot}\omega) :=\inf_{t\in\R} \xi(\theta_t\omega)>-\infty$} for a.e. $\omega\in\Omega$,  and
 $\xi^\omega(t):=\xi(\theta_t\omega)$ is locally H\"older continuous.  }

 \medskip

 It  can be proved that \cite{Arn}
 \begin{equation}
 \label{random-equilibrium-1}
 Y(\omega)=\frac{1}{\int_{-\infty}^0 e^{ s+\int_0^s \xi(\theta_\tau\omega)d\tau}ds}
 \end{equation}
  is a spatially homogeneous asymptotically stable
 random equilibrium of \eqref{real-noise-eq} (see Theorem \ref{real-noise-tm1}),  and the following can also be proved.

 \medskip

 \noindent {\bf v)} {\it For any $\underline{c}\ge 2\sqrt{1+\underline{\xi}}$,
  \eqref{real-noise-eq} has a random transition wave solution $u(x,t)=U(x-C(t;\omega,\mu);\theta_t\omega)$ connecting
  $u=Y(\theta_t\omega)$ and $u\equiv 0$ with least mean speed $\underline{c}$, and  \eqref{real-noise-eq} has no random transition wave solution  connecting
  $u=Y(\theta_t\omega)$ and $u\equiv 0$  with least mean speed less than $2\sqrt{1+\underline{\xi}}$} (see Corollary \ref{real-noise-cor}).

\medskip

\medskip

Third, it is interesting to study front propagation dynamics of \eqref{Main-eq} with {\bf (H1)} being replaced by the following weaker assumption,
\medskip

\noindent {\bf (H1$)'$} $0<\hat a:=\int_\Omega a(\omega)d\PP(\omega)<\infty$.

\medskip

\noindent We plan to study this general case somewhere else, which would have applications to the study of the front propagation dynamics of
the following stochastic Fisher-KPP equation,
\begin{equation}
\label{white-noise-eq}
d u=(u_{xx}+u(1-u))dt+\sigma u dW_t,\quad x\in\R,
\end{equation}
where $W_t$ denotes the standard two-sided Brownian motion ($dW_t$ is then the white noise).
In fact, Let $ \Omega:=\{\omega\in C(\R,\R)\ |\  \omega(0)=0\ \}$ equipped with the open compact topology, $\mathcal{F}$ be the Borel $\sigma-$field and $\mathbb{P}$ be the Wiener measure on $(\Omega, \mathcal{F})$. Let $W_t$ be the one dimensional Brownian motion on the Wiener space $(\Omega,\mathcal{F},\mathbb{P})$ defined by $W_t(\omega)=\omega(t)$. Let $\theta_t\omega$ be the canonical Brownian shift: $(\theta_t\omega)(\cdot)=\omega(t+\cdot)-\omega(t)$ on $\Omega$. It is easy to see that $W_t(\theta_s\omega)=W_{t+s}(\omega)-W_s(\omega)$.
  If $\frac{\sigma^2}{2}<1$,
then it can be proved that
\begin{equation}
\label{random-equilibrium-eq2}
Y(\omega)=\frac{1}{\int_{-\infty}^0 e^{(1-\frac{\sigma^2}{2})s+\sigma W_s (\omega)ds}}
\end{equation}
 is a spatially homogeneous stationary solution process of \eqref{white-noise-eq}. Let $\tilde u=\frac{u}{Y(\theta_t\omega)}$ and drop the tidle, \eqref{white-noise-eq} becomes
\eqref{Main-eq} with $a(\theta_t\omega)= Y(\theta_t\omega)$. The reader is referred to \cite{HuLi1, HuLi2, HuLiWa, JiJiSh, OkVaZh1, OkVaZh2}
for some study on the front propagation dynamics of \eqref{random-equilibrium-eq2}.

The rest of the paper is organized as follows. In section 2, we present some preliminary lemmas and recall some results established in the first part of the series.
 We investigate the existence of random transition fronts of \eqref{Main-eq} and the stability of  random transition fronts of \eqref{Main-eq}
 in sections 3 and 4, respectively.
We consider transition fronts of \eqref{nonautonomous-eq} in section 5.

\section{Preliminary}

In this section, we present some preliminary lemmas to be used in later sections and recall some of the main results proved in the first part.

First, we present some preliminary lemmas.

\begin{lem}
\label{prelim-lm1} (\cite[Lemma 2.1]{SaSh1})
  {\bf (H1)} implies that $\underline{a}(\cdot),a(\cdot),\hat a(\cdot)\in  L^1 (\Omega, \mathcal{F},\PP)$ and that
$\underline{a}(\omega)$ and $\bar a(\omega)$ are independent of $\omega$ for a.e. $\omega\in\Omega$.
\end{lem}

\begin{lem}\label{prelim-lm2} (\cite[Lemma 2.2]{SaSh1})  Suppose that $b\in C(\R, (0,\infty))$  and that $0<\underline{b}\leq \overline{b}<\infty$,  where
$$
\underline{b}=\liminf_{t-s\to\infty}\frac{1}{t-s}\int_s^t b(\tau)d\tau,\quad \bar b=\limsup_{t-s\to\infty}\frac{1}{t-s}\int_s^t b(\tau)d\tau,
$$ then
\begin{equation}
\underline{b}=\sup_{B\in W^{1,\infty}_{\rm loc}(\R)\cap L^{\infty}(\R)}{\rm essinf_{\tau\in\R}(b(\tau)-B'(\tau))}.
\end{equation}
\end{lem}

Let  $b\in C(\R, (0,\infty))$ be given and satisfy  that $0<\underline{b}\leq \overline{b}<\infty$. Consider
\begin{equation}
\label{b-eq1}
u_t=u_{xx}+b(t)u(1-u),\quad x\in\R.
\end{equation}
For given $u_0\in C_{\rm unif}^b(\R)$ with $u_0\ge 0$, let $u(t,x;u_0,b)$ be the solution of \eqref{b-eq1} with $u(0,x;u_0,b)=u_0(x)$.

For every  $0<\mu<\underline{\mu}^*:=\sqrt{\underline{b}}$ and $t\in\R$, let
\begin{equation}
\label{b-eq2}
c(t;b,\mu)=\frac{\mu^2+b(t)}{\mu},\quad
C(t;b,\mu)=\int_0^t c(\tau;b,\mu)d\tau,
\end{equation}
and
\begin{equation}
\label{b-eq3}
\phi^{\mu}(t,x;b)=e^{-\mu (x-C(t;b,\mu))}.
\end{equation}
Then the function $ \phi^{\mu}$ satisfies
\begin{equation}\label{b-eq4}
\phi^{\mu}_t=\phi^{\mu}_{xx} +b(t)\phi^{\mu},\quad x\in\R.
\end{equation}

\begin{lem}
\label{prelim-lm3} (\cite[Lemma 2.3]{SaSh1})
Let
$$
\phi_+^\mu(t,x;b)=\min\{1,\phi^\mu(t,x;b)\}.
$$
Then
$$
u(t,x;\phi_+^\mu(0,\cdot;b),b)\le \phi_+^\mu(t,x;b)\quad \forall\,\, t>0,\,\, x\in\R.
$$
\end{lem}

\begin{lem}\label{prelim-lm4} (\cite[Lemma 2.4]{SaSh1})
   For every $\mu$ with $0<\mu<\tilde{\mu}<\min\{2\mu, \underline{\mu}^*\}$,  there exist  $\{t_k\}_{k\in\Z}$  with $t_k<t_{k+1}$ and $\lim_{k\to\pm\infty}t_k=\pm\infty$,   $B_b\in W^{1,\infty}_{\rm loc}(\R)\cap L^{\infty}(\R)$ with $B_b(\cdot)\in C^1((t_k,t_{k+1}))$ for $k\in\Z$,  and a positive real number $d_b$ such that for every $d\geq d_{b}$ the function
  $$\phi^{\mu,d,B_b}(t,x):=e^{-\mu (x-C(t;b,\mu))}-de^{\big(\frac{\tilde{\mu}}{\mu}-1\big)B_b(t)-\tilde{\mu}(x-C(t;b,\mu))} $$
 satisfies
 $$
 \phi^{\mu,d,B_b}_t\le \phi^{\mu,d,B_b}_{xx}+b(t)\phi^{\mu,d,B_b}(1-\phi^{\mu,d,B_b})
$$
for  $t\in (t_k,t_{k+1})$, $x\ge C(t,b,\mu)+ \frac{\ln d}{\tilde \mu-\mu}+\frac{ B_b(t)}{\mu},\,\, k\in\Z$.
 \end{lem}

\begin{lem}\label{prelim-lm6}(\cite[Lemma 5.1]{SaSh1})
For every $0<\mu<\underline{\mu}^*$, the following hold.
$$\lim_{x\to -\infty} u(t,x+C(t,b(\cdot+t_0),\mu); \phi_+^\mu(0,\cdot;b(\cdot+t_0)),b(\cdot+t_0))=1 \text{ uniformly in}\ t>0 \ \text{ and }\ t_0\in\R,$$ and
$$\lim_{x\to \infty} u(t,x+C(t,b(\cdot+t_0),\mu);\phi_+^\mu(0,\cdot;b(\cdot+t_0)),b(\cdot+t_0))=0\ \text{ uniformly in}\ t>0\ \text{ and}\ t_0\in\R.$$
\end{lem}

\begin{lem}\label{prelim-lm5} (\cite[Lemma 2.9]{SaSh1})
Let $f : \R\times\Omega \to (0,1) $ be a measurable function such that for every  $\omega\in\Omega$ the function $f^{\omega}:=f(\cdot,\omega) : \R \to (0,1)$ is  continuously differentiable and  strictly decreasing. Assume that $\lim_{x\to-\infty}f^{\omega}(x)=1$ and $\lim_{x\to\infty}f^{\omega}(x)=0$ for every $\omega\in\Omega$.  Then for every $a\in (0,1)$ the function $\Omega\ni\omega\mapsto f^{\omega,-1}(a)\in\R$ is measurable, where $f^{\omega,-1}$ denotes the inverse function of $f^{\omega}$.
\end{lem}

 Next, we recall some results about the stability of the positive constant equilibrium solution $u\equiv 1$ of \eqref{Main-eq} (resp.  \eqref{nonautonomous-eq}) established.
Observe that $u(t,x)=v(t,x-C(t;\omega))$ with $C(t;\omega)$ being differential in $t$  solves \eqref{Main-eq} if and only if $v(t,x)$ satisfies
\begin{equation}
\label{moving coordinate -eq1}
v_t=v_{xx}+c(t;\omega)v_x+a(\theta_t\omega)v(1-v),
\end{equation}
where $c(t;\omega)=C'(t;\omega)$.
The following theorems are proved in the first part.

\begin{tm}\label{stability of const equi solu thm} (\cite[Theorem 1.1]{SaSh1})
 Assume (H1) holds.
 \begin{itemize}
 \item[(1)]  For every $0<\tilde{a}<\underline{a}$,  $u_0\in C^{b}_{\rm uinf}(\R)$ with $\inf_{x}u_0(x)>0$ and for almost all $\omega\in\Omega$, there is positive constant $M>0$  such that
$$
\|u(t,\cdot;u_0,\theta_{t_0}\omega)-1\|_{\infty}\leq Me^{-\tilde{a}t},\quad \forall \ t\geq 0, \ t_0\in\R.
$$

\item[(2)]  Let $v(t,x;u_0,\omega)$ be the solution of \eqref{moving coordinate -eq1} with $v(0,x;u_0,\omega)=u_0(x)$. The result in (1)  also holds for $v(t,x;u_0,\omega)$.
\end{itemize}
 \end{tm}

\begin{tm}
\label{uniform-tail-tm} (\cite[Theorem 3.1]{SaSh1})
Assume {\bf (H1)}.
Suppose that $v(t,x;\omega)$ is an entire solution of \eqref{moving coordinate -eq1},  $v(t,x;\omega)$ is nonincreaing in $x$, $0<v(t,x;\omega)<1$. For given $\omega\in\Omega$ with $0<\underline{c}(\cdot;\omega)\le\overline{c}(\cdot;\omega)<\infty$, if  there is $x^*\in\R$ such that $\inf_{t\in\R} v(t,x^*;\omega)>0$, then $\lim_{x\to -\infty} v(t,x;\omega)=1$ uniformly in $t\in\R$.
\end{tm}

\begin{tm}
\label{real-noise-tm1} (\cite[Theorem 3.2, Lemma 3.2, Corollary 3.1]{SaSh1})

Assume (H3). $Y(\omega)=\frac{1}{\int_{-\infty}^0 e^{ s+\int_0^s \xi(\theta_\tau\omega)d\tau}ds}$ is a random equilibrium of \eqref{real-noise-eq}.
Moreover, the following hold.
\begin{itemize}
\item[(1)] For each $\omega\in\Omega$, { $0<Y_{\inf}(\theta_\cdot\omega)\leq Y_{\sup}(\theta_\cdot\omega)<\infty$.}

\item[(2)] For each $\omega\in\Omega$, $\lim_{t\to\infty} \frac{\ln Y(\theta_t\omega)}{t}=0$ and   $\lim_{t\to\infty} \frac{\int_0^t Y(\theta_s\omega)ds}{t}=1$.

\item[(3)] $\underline{Y}(\omega)=1-\underline{\xi}>0$ and $\overline{Y}(\omega)-1+\overline{\xi}<\infty$ for a.e. $\omega\in\Omega$.

\item[(4)] For given  $u_0\in C^{b}_{\rm uinf}(\R)$ with $\inf_{x}u_0(x)>0$,  for a.e. $\omega\in\Omega$,
\begin{equation*}
\lim_{t\to\infty}\|\frac{u(t,\cdot;u_0,\theta_{t_0}\omega)}{Y(\theta_t\theta_{t_0}\omega)}-1\|_{\infty}=0
\end{equation*}
 uniformly in $t_0\in\R$ where $u(t,x;u_0,\theta_{t_0}\omega)$ is the solution of \eqref{real-noise-eq} with $u(0,x;u_0,\theta_{t_0}\omega)=u_0(x)$.
 \end{itemize}
 \end{tm}

Now, we recall some results on the spreading speeds of \eqref{Main-eq} and \eqref{nonautonomous-eq} from the first part.

Let  $\underline{c}^*$, $\hat c^*$, and $\bar c^*$ be as in \eqref{minimal-speeds-eq}. Let
 $$
 X_c^+=\{u\in C_{\rm unif}^b(\R)\,|\, u\ge 0,\,\,  {\rm supp}(u)\,\,\, \text{is bounded and not empty}\}.
 $$

 \begin{defin}
 \label{spreading-speed-def}
 For given $\omega\in\Omega$, let
 $$
 C_{\sup}(\omega)=\{c\in\R^+\,|\, \limsup_{t\to\infty}\sup_{s\in\R,|x|\ge ct}u(t,x;u_0,\theta_s\omega)=0\quad \forall\,\, u_0\in X_c^+\}
 $$
 and
 $$
 C_{\inf}(\omega)=\{c\in\R^+\,|\, \limsup_{t\to\infty}\sup_{s\in\R,|x|\le ct} |u(t,x;u_0,\theta_s\omega)-1|=0\quad \forall\,\, u_0\in X_c^+\}.
 $$
 Let
 $$
 c_{\sup}^*(\omega)=\inf\{c\,|\, c\in C_{\sup}(\omega)\},\quad c_{\inf}^*(\omega)=\sup\{c\,|\, c\in C_{\inf}(\omega)\}.
 $$
 $[c_{\inf}^*(\omega),c_{\sup}^*(\omega)]$ is called the {\rm spreading speed interval} of \eqref{Main-eq}  with respect to compactly supported initial functions.
 \end{defin}

\begin{tm}\label{spreading-speeds-tm} (\cite[Theorem 1.2]{SaSh1})
Assume that {\bf (H1)} holds. Then the following hold.

\begin{itemize}
\item[(i)] For any $\omega\in\Omega_0$, $c_{\sup}^*(\omega)=\bar{c}^*$.

\item[(ii)]  For any $\omega\in\Omega_0$,  $c_{\inf}^{*}(\omega)=\underline{c}^*$.
\end{itemize}
\end{tm}

Let
 $$
\tilde  X_c^+=\{u\in C_{\rm unif}^b(\R)\,|\, u\ge 0,\,\,  \liminf_{x\to -\infty}u_0(x)>0,\,\, u_0(x)=0\,\, {\rm for}\,\, x\gg 1\}.
 $$

 \begin{defin}
 \label{spreading-speed-def1}
 For given $\omega\in\Omega$, let
 $$
 \tilde C_{\sup}(\omega)=\{c\in\R^+\,|\, \limsup_{t\to\infty}\sup_{s\in\R,x\ge  ct}u(t,x;u_0,\theta_s\omega)=0\quad \forall\,\, u_0\in \tilde X_c^+\}
 $$
 and
 $$
 \tilde C_{\inf}(\omega)=\{c\in\R^+\,|\, \limsup_{t\to\infty}\sup_{s\in\R,x\le ct} |u(t,x;u_0,\theta_s\omega)-1|=0\quad \forall\,\, u_0\in \tilde X_c^+\}.
 $$
 Let
 $$
 \tilde c_{\sup}^*(\omega)=\inf\{c\,|\, c\in \tilde C_{\sup}(\omega)\},\quad \tilde c_{\inf}^*(\omega)=\sup\{c\,|\, c\in \tilde C_{\inf}(\omega)\}.
 $$
 $[\tilde c_{\inf}^*(\omega),\tilde c_{\sup}^*(\omega)]$ is called the {\rm spreading speed interval} of \eqref{Main-eq}  with respect to front-like initial functions.
 \end{defin}
Let $u_0^*(\cdot)$ be the function satisfying that
$u_0^*(x)=1$ for $x< 0$ and $u_0^*(x)=0$ for $x>0$.  Note that $u(t,x;u_0^*,\omega)$ exists (see \cite[Theorem 1]{KPP}).

\begin{tm}\label{spreading-speeds-tm-0}
Assume that {\bf (H1)} holds. Then the following hold.
\begin{itemize}
\item[(1)] (\cite[Theorem 1.3]{SaSh1}) For any $\omega\in\Omega_0$, $\tilde c_{\sup}^*(\omega)=\bar{c}^*$ and  $\tilde c_{\inf}^{*}(\omega)=\underline{c}^*$.

\item[(2)] (\cite[Theorem 1.4]{SaSh1})     For a.e. $\omega\in\Omega$,
\begin{equation}
\label{average-speed-eq}
\lim_{t\to\infty}\frac{x(t,\omega)}{t}= \hat c^*,
\end{equation}
where $x(t,\omega)$ is such that $u(t,x(t,\omega);u_0^*,\omega)=\frac{1}{2}$.  Moreover,
 for any $u_0\in \tilde X_c^+$, it holds that
\begin{equation}\label{asymptoc-tail-eq3}
\lim_{t\to\infty}\sup_{x\geq(\hat c^*+h)t}u(t,x;u_0,\omega)=0, \forall\ h>0, \ \text{a.e }\ \omega
\end{equation}
and
\begin{equation}\label{asymptoc-tail-eq4}
\lim_{t\to\infty}\inf_{x\leq (\hat c^*-h)t}u(t,x;u_0,\omega)=1, \forall\ h>0, \ \text{a.e\  }\ \omega.
\end{equation}
\end{itemize}
\end{tm}

\section{Existence of random transition fronts}

In this section, we study the existence of random traveling wave solutions of \eqref{Main-eq}.

The main results of this section are stated in the following two theorems.

\begin{tm}
\label{average-speed-thm} Assume that {\bf (H1)} holds.
Let $u(t,x;U(\cdot,\omega),\omega)=U(x-C(t,\omega),\theta_t\omega)$ be a monotone transition front solution of \eqref{Main-eq}. Then there is
$\hat c\ge 2\sqrt {\hat a}$ such that
 \begin{equation}
\lim_{t\to\infty}\frac{C(t,\omega)}{t}=\hat {c},\quad \text{for a.e }\ \omega\in\Omega.
\end{equation}
\end{tm}

 \begin{tm}\label{Existence of random transition front} Assume that {\bf (H1)} holds.
 \begin{itemize}
 \item[(1)]  For any given $c > \underline{c}^*$, there is a monotone
random transition front of \eqref{Main-eq} with least mean speed $\underline{c}=c$. In particular, for given $c>\underline{c}^*$, let  $\mu\in(0, \sqrt{\underline{a}})$ be such that $\underline{c}=\frac{\mu^2+\underline{a}}{\mu}$. Then
      \eqref{Main-eq} has a monotone random transition wave solution $u(t,x)=U^\mu(x-C(t;\omega,\mu),\theta_t\omega)$ with  $C(t;\omega,\mu)=\int_0^t c(s;\omega,\mu)ds$, where
\begin{equation}\label{Exist-tm-eq1}
c(t;\omega,\mu)=\frac{\mu^2+a(\theta_t\omega)}{\mu},
 \end{equation}
and hence  $\hat c=\frac{\mu^2+\hat a}{\mu} > \hat c^*$ and $\underline{c}=\frac{\mu^2+\underbar a}{\mu} >\underline{c}^*$.
Moreover, for any $\omega\in\Omega_0$,
\begin{equation}\label{Exist-tm-eq2}
\lim_{x\to \infty}\sup_{t\in\R}\Big|\frac{U^\mu(x,\theta_t\omega)}{e^{-\mu x}}-1\Big|=0 \quad \text{and}\quad \lim_{x\to -\infty}\sup_{t\in\R}|U^\mu(x,\theta_t\omega)-1|=0.
\end{equation}

\item[(2)] There is a monotone random transition front of \eqref{Main-eq}  with least mean speed  $\underline{c}^*$.

\item[(3)] There is no random transition front of \eqref{Main-eq}  with  least mean  speed less than $\underline{c}^*$.
\end{itemize}
\end{tm}

\begin{rk}  Observe that the function $(0,\sqrt{\underline{a}})\ni\mu \mapsto \frac{\mu^2+\hat{a}}{\mu}$ is continuous and decreasing with $\lim_{\mu\to \sqrt{\underline{a}}^-}\frac{\mu^2+\hat{a}}{\mu}=\frac{\underline{a}+\hat{a}}{\sqrt{\underline{a}}}$. Hence it follows from Theorem \ref{Existence of random transition front} (1) that for every $c>\frac{\underline{a}+\hat{a}}{\sqrt{\underline{a}}}(\ge 2\sqrt{\hat a}=\hat c^*)$, \eqref{Main-eq} has a transition wave solution $u(t,x)=U^\mu(x-C(t;\omega,\mu),\theta_t\omega)$ with average mean $c$. We refer the random transition front in  Theorem \ref{Existence of random transition front} (2) as the critical transition front of
\eqref{Main-eq}. It follows from Theorem \ref{average-speed-thm} and the proof of Theorem \ref{Existence of random transition front} (2) (mainly inequality \eqref{upper-bound-eq-for-ave-mean}) that the average speed, $\hat{c}_{\rm crit}$, of the critical random transition front exists and $\hat{c}_{\rm crit}\in [2\sqrt{\hat a}, \frac{\underline{a}+\hat{a}}{\sqrt{\underline{a}}}]$. It can also be shown that the critical random transition front has the minimum average speed. It remains open whether $\hat{c}_{\rm crit}=2\sqrt{\hat{a}}$. It also remains open whether \eqref{Main-eq} has random transition fronts with average speed $\hat c\in (\hat{c}_{\rm crit}, \frac{\underline{a}+\hat{a}}{\sqrt{\underline{a}}}]$ if this interval is not empty.
\end{rk}

\subsection{Proof of Theorem \ref{average-speed-thm}}

In this subsection, we give a proof of Theorem \ref{average-speed-thm}

\begin{proof} [Proof of Theorem \ref{average-speed-thm}]
Suppose $u(t,x;\omega)=U(x-C(t;\omega),\theta_t\omega)$ is a monotone transition front of \eqref{Main-eq}.

First, we claim that for a.e. $\omega\in\Omega$,
\begin{equation}
\label{average-speed-eq1}
 C(t+s;\omega)=C(t;\omega)+C(s;\theta_t\omega), \quad \forall\ s\ge 0,\ t\in\R.
 \end{equation}
Indeed, for a.e. $\omega\in\Omega$, we have
\begin{align*}
U(x-C(s;\theta_t\omega),\theta_{t+s}\omega)=& u(s,x;U(\cdot,\theta_t\omega),\theta_t\omega)\\
=&u(s,x;u(t,\cdot+C(t,\omega);U(\cdot,\omega),\omega),\theta_t\omega)\\
=&u(t+s,x+C(t;\omega);U(\cdot,\omega),\omega)\\
=& U(x+C(t;\omega)-C(t+s;\omega),\theta_{t+s}\omega), \quad \forall\ x\in\R, \ s\ge 0,\ t\in\R.
\end{align*}
Since the function $x\mapsto U(x,\theta_{t+s}\omega)$ is decreasing for a.e $\omega\in\Omega$, then we have that for a.e. $\omega\in\Omega$,
$$
x-C(s;\theta_t\omega)=x+C(t;\omega)-C(t+s;\omega), \quad \forall\ s,t>0. $$
This implies that \eqref{average-speed-eq1} holds.

Next, by \eqref{average-speed-eq1} and  subadditive  ergodic  theorem, there is $\hat c\in\R$ such that for a.e.
$\omega\in\Omega$, there holds
\begin{equation}
\label{average-speed-eq2}
\lim_{t\to\infty}\frac{C(t;\omega)}{t}=\hat c.
\end{equation}

We show now $\hat c\ge 2\sqrt {\hat a}$. By Theorem \ref{spreading-speeds-tm-0},
$$
1=\lim_{t\to\infty}u(t,(2\sqrt{\hat{a}}-\varepsilon)t;U(\cdot;\omega),\omega)\leq \liminf_{t\to\infty}U((2\sqrt{\hat{a}}-\varepsilon)t-C(t,\omega),\theta_t\omega).
$$
Note that $\lim_{x\to \infty}U(x,\theta_t\omega)=0$ uniformly in $t\in\R$. Hence there is $M(\omega)>0$ such that
$$
(2\sqrt{\hat{a}}-\varepsilon)t\leq C(t,\omega) +M(\omega), \quad \forall\ t\gg 1.
$$
This implies that
$$
2\sqrt{\hat{a}}-\varepsilon\leq \lim_{t\to\infty}\frac{C(t,\omega)}{t}=\tilde{c}^*, \quad \forall\ \varepsilon>0.
$$
Letting $\varepsilon\to0$, we have $\hat c\ge 2\sqrt {\hat a}$.
\end{proof}

\subsection{Proof of Theorem \ref{Existence of random transition front} (1)}

In this subsection, we present the proof of Theorem \ref{Existence of random transition front} (1).
We first prove some lemmas.   We shall always suppose that  {\bf (H1)}  holds.

Let $\Omega_0$ be as in \eqref{omega-0}. Hence
$$
\lim_{t\to\infty}\frac{1}{t}\int_0^t a(\theta_\tau\omega)d\tau=\int_\Omega a(\omega) d\P(\omega),\quad 0<\underline{a}(\omega)\le \overline{a}(\omega)<\infty, \quad \forall\,\omega\in\Omega_0.
$$

Recall that, if  $u(t,x)=v(t,x-C(t;\omega))$ with $C(t;\omega)$ being differential in $t$  solves \eqref{Main-eq}, then  $v(t,x)$ satisfies
\eqref{moving coordinate -eq1}.
Hence, to prove the existence of random transition fronts of \eqref{Main-eq} of the form
  $u(t,x)=U(x-C(t;\omega),\theta_t\omega)$ for some differentiable $C(t;\omega)$ and some  $U(x,\omega)$ which is measurable in $\omega$ and $U(-\infty,\theta_t\omega)=1$  and $U(\infty;\theta_{t}\omega)=0$ uniformly in t,
  it is equivalent
to prove  the existence of entire solutions of \eqref{moving coordinate -eq1} (with $c(t;\omega)=C^{'}(t;\omega)$) of the form  $v(t,x)=V(t,x;\omega)$ such that
\begin{equation}\label{random entire v sol def}
\begin{cases}
\omega\mapsto V(t,x;\omega),\quad \omega\mapsto C(t;\omega)\quad \text{are measurable},\\
V(t,x;\omega)=V(0,x;\theta_t\omega),\\
\lim_{x\to-\infty}V(t,x;\omega)=1, \quad \text{and} \quad \lim_{x\to\infty}V(t,x;\omega)=0 \quad \text{uniforlmy in } t.
\end{cases}
\end{equation}

For every  $0<\mu<\underline{\mu}^*:=\sqrt{\underline{a}}$, $x\in\R$, $t\in\R$ and $\omega\in\Omega$, let
$$
\phi^{\mu}(x)=e^{-\mu x} \,\, \text{and}\,\, C(t;\omega,\mu)=\int_0^tc(\t,\omega,\mu)d\t \,\, \text{with} \,\, c(t;\omega,\mu)=\frac{\mu^2+a(\theta_t\omega)}{\mu}.
$$
Then the function $ \phi^{\mu}$ satisfies
\begin{equation}\label{Linearized moving coordinate eq at 0}
\phi^{\mu}_t=\phi^{\mu}_{xx}+c(t;\omega,\mu)\phi^{\mu}_{x} +a(\theta_{t}\omega)\phi^{\mu},\quad x\in\R.
\end{equation}
Since $a(\omega)>0$ for every $\omega\in\Omega$, it follows from \eqref{Linearized moving coordinate eq at 0} that $v(t,x)= \phi^{\mu}(x)$ is a super-solution of \eqref{moving coordinate -eq1} with $c(t;\omega)=c(t;\omega,\mu)$.  We also note that $v(x,t)\equiv 1$ is a solution of \eqref{moving coordinate -eq1}. We  introduce the functions,
\begin{equation}\label{super-sol}
\phi_{+}^{\mu}(x):=\min\{1, \phi^{\mu}(x)\}, \quad \forall\ \, x\in\R
\end{equation}
and
\begin{equation}\label{G-def}
\mathcal{G}^{\omega,\mu}(v)(t,x):=v_t-v_{xx}-c(t;\omega,\mu)v_x-a(\theta_t\omega)v(1-v).
\end{equation}

The following Lemma will be frequently used to prove our main results.

\begin{lem}\label{sub-sol-lem}
 Suppose that {\bf (H1)} holds. Let $\omega\in\Omega_0$. Then for every $0<\mu<\tilde{\mu}<\min\{2\mu, \underline{\mu}^*\}$,  there exist  $\{t_k\}_{k\in\Z}$  with $t_k<t_{k+1}$ and $\lim_{k\to\pm\infty}t_k=\pm\infty$,   $A_{\omega}\in W^{1,\infty}_{\rm loc}(\R)\cap L^{\infty}(\R)$ with $A_\omega(\cdot)\in C^1((t_k,t_{k+1}))$ for $k\in\Z$,  and a positive real number $d_{\omega}$ such that for every $d\geq d_{\omega}$
   the function
  $$\phi^{\mu,d, A_\omega}(t,x):=e^{-\mu x}-de^{\big(\frac{\tilde{\mu}}{\mu}-1\big)A_{\omega}(t)-\tilde{\mu}x} $$
 satisfies
 $$
 \mathcal{G}^{\omega,\mu}(\phi^{\mu,d,A_\omega})(t,x)\leq 0\quad {\rm for}\quad t\in (t_k,t_{k+1}),\,\, x\ge \frac{\ln d}{\tilde \mu-\mu}+\frac{ A_\omega(t)}{\mu},\,\, k\in\Z.
 $$
 \end{lem}

 \begin{proof} It follows from Lemma \ref{prelim-lm4}.
 \end{proof}

 Suppose that  {\bf (H1)} holds.
Let $\omega \in\Omega_0$, and $0<\mu<\tilde{\mu}<\min\{2\mu,\underline{\mu}^*\}$ be given. Let  $A_{\omega}$ and  $d_\omega$   be given by Lemma \ref{sub-sol-lem}. Let
\begin{equation}
\label{x-omega-eq}
x_\omega(t)=\frac{\ln d_{\omega}+\ln \tilde \mu-\ln \mu}{\tilde \mu-\mu}+ \frac{A_\omega(t)}{\mu}.
\end{equation}
Note that for any given $t\in\R$,
$$
\phi^{\mu,d_\omega,A_\omega}(t,x_\omega(t))=\sup_{x\in\R} \phi^{\mu,d_\omega,A_\omega}(t,x)=e^{-\mu\big(\frac{\ln d}{\tilde \mu-\mu}+\frac{A_\omega(t)}{\mu}\big)}e^{-\mu \frac{\ln \tilde \mu-\ln \mu}{\tilde\mu-\mu}}\big(1-\frac{\mu}{\tilde \mu}\big).
$$
 We introduce the following function
\begin{equation}\label{Lower -sol}\phi_{-}^{\mu}(t,x; \theta_{t_0}\omega)=\begin{cases}
{ \phi^{\mu,d_{\omega},A_\omega}( t+t_0,x)},\quad  \text{if}\ x\geq x_{\omega}(t+t_0), \\
\phi^{\mu,d_{\omega},A_\omega}( t+t_0,x_\omega( t+t_0)),\quad \ \text{if}\ x\leq x_{\omega}( t+t_0).
\end{cases}
\end{equation}
It is clear  that
$$
0<\phi_{-}^{\mu}(t,x; \theta_{t_0}\omega)< \phi_{+}^{\mu}(x)\leq 1, \forall\ t\in\R,\,\, x\in\R,\,\, t_0\in\R.
$$
and
\begin{equation}\label{decay-rate-eq1}
\lim_{x\to\infty}\sup_{t>0,t_0\in\R}|\frac{\phi_{-}^{\mu}(t,x;\theta_{t_0}\omega)}{\phi_+^\mu(x)}-1|=0.
\end{equation}
\begin{lem}
\label{lm00}
For every $0<\mu<\underline{\mu}_*$,
$$
\phi_{-}^{\mu}(t,x;\omega)\leq u(t,x+C(t;\omega,\mu);\phi_+^\mu,\theta)\le \phi_+^\mu(x)\quad \forall\,\, t>0,\,\, x\in\R.
$$
\end{lem}
\begin{proof}
It follows from Lemmas \ref{prelim-lm3} and \ref{sub-sol-lem} and comparison principle for parabolic equations.
\end{proof}
\begin{lem}
\label{lm01}
For every $\omega\in\Omega_0$,
$\lim_{x\to -\infty} u(t,x+C(t,\theta_{t_0}\omega,\mu); \phi_+^\mu(0,\cdot;\theta_{t_0}\omega),\theta_{t_0}\omega)=1$ uniformly in $t>0$ and $t_0\in\R$, and
$\lim_{x\to \infty} u(t,x+C(t,\theta_{t_0}\omega,\mu);\phi_+^\mu(0,\cdot;\theta_{t_0}\omega),\theta_{t_0}\omega)=0$ uniformly in $t>0$ and $t_0\in\Omega$.
\end{lem}
\begin{proof}
It follows from Lemma \ref{prelim-lm6}
\end{proof}
Now, we present the proof of  Theorem \ref{Existence of random transition front}(1).

\medskip

\begin{proof}
First, let $0<\mu<\tilde{\mu}<\min\{2\mu,\underline{\mu}^*\}$ be fixed. It follows from Lemma \ref{lm00} that
$$
u(t,x+C(t;\omega,\mu);\phi_\mu^+,\omega)<u(\tilde{t},x+C(\tilde{t};\omega,\mu);\phi_\mu^+,\omega), \quad \forall\ x\in\R,\ t>\tilde{t}>0, \ \forall\ \omega\in\Omega_0.
$$
Hence the following limit exits
\begin{equation}\label{U-mu-def1}
U^{\mu}(x,\omega):=\lim_{t\to\infty}u(t,x+C(t;\theta_{-t}\omega,\mu);\phi_\mu^+,\theta_{-t}\omega), \quad \forall\ x\in\R, \ \omega\in\Omega_0.
\end{equation}
Furthermore, it follows from Lemma \ref{lm01} that  for every $\omega\in\Omega_0$ the limit in \eqref{U-mu-def1} is uniform in $x\in\R$ and
$$
\lim_{x\to-\infty}U^\mu(x,\theta_t\omega)=1 \quad \text{and}\ \lim_{x\to\infty}U^\mu(x,\theta_t\omega)=0, \ \text{uniformly in }\ t\in\R.
$$

Next, using the fact that $C(t+\tau;\theta_{-\t}\omega,\mu)=C(\t;\theta_{-\t}\omega)+C(t;\omega,\mu)$, we have
\begin{align*}
u(t,x+C(t,\omega,\mu);U^\mu(\cdot,\omega),\omega)=&\lim_{\t\to\infty}u(t,x+C(t,\omega,\mu);u(\t,x+C(\t;\theta_{-\t}\omega,\mu);\phi_\mu^+,\theta_{-\t}\omega),\omega)\cr
=&\lim_{\t\to\infty}u(t+\tau,x+C(t,\omega,\mu)+C(\t;\theta_{-\t}\omega,\mu);\phi_\mu^+,\theta_{-\t}\omega)\cr
=&\lim_{\t\to\infty}u(t+\tau,x+C(t+\t,\theta_{-(t+\t)}\theta_t\omega,\mu);\phi_\mu^+,\theta_{-(t+\t)}\theta_t\omega)\cr
=&U^\mu(x,\theta_t\omega).
\end{align*}
It follows from \eqref{decay-rate-eq1} and Lemma \ref{lm00} that
$$
\lim_{x\to\infty}\sup_{t\in\R}|\frac{U^\mu(x,\theta_t\omega)}{\phi_+^\mu(x)}-1|=0
$$  Furthermore, since the function $\R\ni x\mapsto \phi_+^\mu$ is non-increasing, then for every $\omega_0\in\Omega$ and every $t> 0$, we have that the function $ \R\ni x\mapsto u(t,x+C(t;\theta_{-t}\omega);\phi_+^\mu,\theta{-t}\omega)$ is decreasing, hence so is $U^\mu(\cdot,\omega)$. This completes the proof of the result.
\end{proof}

\subsection{Proof of Theorem \ref{Existence of random transition front}(2)}

In this subsection, we present the proof of Theorem \ref{Existence of random transition front}(2). We first  present some Lemmas. Let

\begin{equation}\label{u-0-def}
u_0^*(x)=\begin{cases}
         1 , \quad x\leq 0\cr
         0 , \quad x>0.
       \end{cases}
\end{equation}
\begin{lem}\label{minimal-front-lem1}
For every $\tilde u_0\in C^b_{\rm unif}(\R)$  satisfying that $0<\tilde{u}_{0}(x)<1$ and
$$
\lim_{x\to-\infty}\tilde{u}_0(x)=1 \quad \text{and}\quad \lim_{x\to\infty}\tilde{u}_0(x)=0,
$$
the following hold.

\begin{enumerate}
\item[(1)] For every $t>0$ and $\omega\in\Omega$, there is $x(t,\omega)\in[-\infty, \infty]$ such that
$$
u(t,x;u_0^*,\theta_{-t}\omega)\begin{cases}>u(t,x;\tilde{u}_0,\theta_{-t}\omega), \ x<x(t,\omega)\cr
<u(t,x;\tilde{u}_0,\theta_{-t}\omega), \ x>x(t,\omega).
\end{cases}
$$
\item[(2)] For given $\omega\in\Omega$, suppose that
$$
U_i(x)=\lim_{n\to\infty}u(t_n,x;u_0^*(\cdot+x_{i,n}(\omega)),\theta_{-t_n}\omega),\quad \forall\ x\in\R \quad i=1,2
$$
exists for some $t_n\to\infty$ and $x_{i,n}(\omega)\in\R$ ($i=1,2$). If $U_1(0)=U_2(0)$ then $U_1(x)=U_2(x)$ for every $x\in\R$.
\end{enumerate}
\end{lem}
\begin{proof}
(1) See \cite[Lemma 4.6 (1)]{She4}.

(2) See \cite[Lemma 4.5 (2)]{She4}.
\end{proof}

\begin{proof}[Proof of Theorem \ref{Existence of random transition front}(2)] The proof of this result is divided into several steps.  Throughout this proof,   $u_0^*$ is given by \eqref{u-0-def}.

By comparison principle for parabolic equations we have that $ 0<u(t,x;u_0^*,\omega)<1$  for every  $t> 0$, $x\in\R$, $\omega\in\Omega$. Furthermore, for every $t>0$ fixed, the function $(x,\omega)\mapsto u(t,x;u_0^*,\omega)$ is measurable in $\omega\in\Omega$, and continously differentiable and strictly decreasing in $x\in\R$ with
$$
\lim_{x\to-\infty}u(t,x;u_0^*,\omega)=1 \quad \text{and}\quad \lim_{x\to\infty}u(t,x;u_0^*,\omega)=0, \quad \forall\,\, \omega\in\Omega.
$$
Hence, by Lemma \ref{prelim-lm5}, for every $t>0$ and $\omega\in\Omega$, there is  $x(t,\omega)\in\R $ measurable in $\omega\in\Omega$ satisfying
\begin{equation}\label{g-eq0}
u(t,x(t,\omega);u_{0}^*,\omega)=\frac{1}{2}.
\end{equation}
Next, consider the functions
$$
\tilde{u}(t,x;\omega)=u(t,x+x(t,\theta_{-t}\omega);u_{0}^*,\theta_{-t}\omega),\quad x\in\R.
$$


\noindent{\bf Step 1.} We claim that for any $\omega\in\Omega$ and $0<t_1<t_2$,
\begin{equation}\label{g-e00}
\tilde{u}(t_1,x,\omega)\begin{cases}
> \tilde{u}(t_2,x;\omega)\ \forall x< 0\cr
< \tilde{u}(t_2,x;\omega)\ \forall x>0.
\end{cases}
\end{equation}

Observe that
$$
\tilde{u}(t_1,x;\omega)=u(t_1,\cdot;u_0^*(\cdot+x(t_1,\theta_{-t_1}\omega)),\theta_{-t_1}\omega)
$$
and
$$
\tilde{u}(t_2,x;\omega)=u(t_1,x;u(t_2-t_1,\cdot;u_0^*(\cdot+x(t_2,\theta_{-t_2}\omega)),\theta_{-t_2}\omega)),\theta_{-t_1}\omega)
$$
Thus  \eqref{g-e00} follows from \eqref{g-eq0} and Lemma \ref{minimal-front-lem1}(1).

Hence, $U(x;\omega)=\lim_{t\to\infty}\tilde{u}(t,x;\omega)$ exists. Moreover, it follows from estimate for parabolic equations that $U(\cdot;\omega)\in C^{b}_{\rm unif}(\R)$ and $\lim_{t\to\infty}\tilde{u}(t,x;\omega)=U(x;\omega)$ locally uniform in $x\in\R$.

\medskip

\noindent{\bf Step 2.} We show that,  { for any $\omega\in\Omega_0$},
\begin{equation}\label{boundary-estimate-of U}
\lim_{x\to-\infty}U(x,\theta_t\omega)=1 \quad \text{and} \quad \lim_{x\to\infty}U(x,\theta_t\omega)=0\quad \text{uniformly in } t\in\R.
\end{equation}
and
\begin{equation}\label{uniform-conv-to-U}
U(x,\omega)=\lim_{t\to\infty} u(t,x+x(t,\theta_{-t}\omega);u_0^*,\theta_{-t}\omega) \quad \text{uniformly in } x\in\R.
\end{equation}

\smallskip

For every $0<\mu<\sqrt{\underline{a}}$, let $U^{\mu}(x,\theta_{t}\omega)=u(t,x+C(t;\omega,\mu);U^\mu(\cdot,\omega),\omega)$, with $C(t;\omega,\mu)=\int_{0}^t\frac{\mu^2+a(\theta_\tau\omega)}{\mu}d\tau $, denote the transition front given by Theorem \ref{Existence of random transition front} (1).  For { $t>0$} and {$\omega\in\Omega_0$},  let $X_{\mu}(t,\omega)\in\R$ be such that
 \begin{equation}\label{mu-speed-def}
u(t,X_{\mu}(t,\omega);{U^{\mu}(\cdot,\theta_{-t}\omega),\theta_{-t}\omega})=U^{\mu}(X_{\mu}(t,\omega)-C(t;{ \theta_{-t}\omega},\mu),{ \omega})=\frac{1}{2}.
\end{equation}
Observe, as above, that $\omega\mapsto X_\mu(t,\omega)$  is mearable in $\omega\in{\Omega_0}$. {For any $\omega\in\Omega_0$}, since $U^{\mu}(-\infty,\theta_t\omega)=1$ and $U^{\mu}(\infty,\theta_t\omega)=0$ uniformly in $t\in\R$ and $U^\mu(x,\omega)$ is strictly decreasing in $x$, there are $M_\mu(\omega)>0$ and  $x_\mu(\omega)$ such that
\begin{equation}\label{g-e2}
\begin{cases}
X_{\mu}(t,\omega)-C(t;{\theta_{-t}\omega},\mu)=x_\mu(\omega)\quad \forall\,\, t\in\R\cr
|x_\mu(\theta_t\omega)|\le M_\mu(\omega)\quad \forall\,\, t\in\R.
\end{cases}
\end{equation}

It follows from \eqref{g-eq0}, \eqref{mu-speed-def}, Lemma \ref{minimal-front-lem1}(1) and the fact that $0<U^\mu(\cdot,\theta_{-t}\omega)<1$ that for every $\omega\in\Omega_0$ and every $t>0$,
\begin{equation*}
u(t,x+x(t,\theta_{-t}\omega);u_0^*,\theta_{-t}\omega)\begin{cases}
> u(t,x+X_{\mu}(t,\omega),U^{\mu}(\cdot,\theta_{-t}\omega),\theta_{-t}\omega),\quad  x<0\cr
< u(t,x+X_{\mu}(t,\omega),U^{\mu}(\cdot,\theta_{-t}\omega),\theta_{-t}\omega),\quad  x>0.
\end{cases}
\end{equation*}
That is
\begin{equation}\label{g-e2-2}
u(t,x+x(t,\theta_{-t}\omega);u_0^*,\theta_{-t}\omega)\begin{cases}
> U^{\mu}(x+{ x_\mu(\omega)},\omega),\quad  x<0\cr
< U^{\mu}(x+{ x_\mu(\omega)},\omega),\quad  x>0.
\end{cases}
\end{equation}
This  yields that
$$
U(x,\omega)\begin{cases}
\geq U^{\mu}(x+x_{\mu}(\omega),\omega),\quad  x\leq 0\cr
\leq U^{\mu}(x+x_{\mu}(\omega),\omega),\quad  x\geq 0,
\end{cases}
$$
and then that,
for any $t\in\R$,
\begin{equation}\label{g-e2-2-1}
U(x,\theta_t\omega)\begin{cases}
\geq U^{\mu}(x+x_{\mu}(\theta_t\omega),\theta_t\omega),\quad  x\leq 0\cr
\leq U^{\mu}(x+x_{\mu}(\theta_t\omega),\theta_t\omega),\quad  x\geq 0.
\end{cases}
\end{equation}
 {\eqref{boundary-estimate-of U} and  \eqref{uniform-conv-to-U} then follow from \eqref{g-e2}, \eqref{g-e2-2},  and \eqref{g-e2-2-1}.}

\smallskip
\smallskip

\noindent{\bf Step 3.} We claim that there is $C(t,\omega)$ measurable in $\omega\in\Omega_0$ satisfying
\begin{equation}
u(t,x+C(t;\omega);U(\cdot,\omega),\omega)\equiv U(x,\theta_t\omega)\quad { \forall\,\, t\in\R,\,\, x\in\R,\,\, \omega\in\Omega_0}.
\end{equation}

In fact, similar arguments leading to the existence of $x(t,\omega)$ above  show that  for every $t$ there is $C(t,\omega)\in\R$ measurable in $\omega\in\Omega_0$ such that { for any $\omega\in\Omega_0$},
\begin{equation}\label{minimal-wave-speed-eq1}
u(t,C(t;\omega);U(\cdot,\omega),\omega)=\frac{1}{2}.
\end{equation}
Observe that for $\omega\in\Omega_0$,
\begin{align*}
u(t,x+C(t;\omega);U(\cdot,\omega),\omega)=& \lim_{s\to\infty}u(t,x+C(t;\omega);u(s,\cdot;u_0^*(\cdot+x(s;\theta_{-s}\omega)),\theta_{-s}\omega),\omega)\cr
=&\lim_{s\to\infty}u(t+s,x;u_0^*(\cdot+x(s;\theta_{-s}\omega)+C(t;\omega)),\theta_{-(s+t)}\theta_t\omega)
\end{align*}
and
\begin{align*}
U(x,\theta_t\omega)=&\lim_{\t\to\infty}u(\t,x;u_0^*(\cdot+x(\t;\theta_{-\t+t}\omega)),\theta_{-\tau}\theta_{t}\omega)\cr
=&\lim_{s\to\infty}u(s+t,x;u_0^*(\cdot+x(s+t;\theta_{-s}\omega)),\theta_{-(s+t)}\theta_t\omega).
\end{align*}
Since,
$$u(t,C(t;\omega);U(\cdot,\omega),\omega)=\frac{1}{2}\quad \text{and}\quad U(0,\theta_t\omega)=\frac{1}{2}, $$
we conclude from Lemma \ref{minimal-front-lem1}(2) that for $\omega\in\Omega_0$,
$$
u(t,x+C(t;\omega);U(\cdot,\omega),\omega)\equiv U(x,\theta_t\omega).
$$

\noindent{\bf Step 4.} In this step, we shall show that {for any $\omega\in\Omega_0$},
\begin{equation}\label{g-e2-2-3}
\liminf_{t-s\to\infty}\frac{C(t;\omega)-C(s;\omega)}{t-s}\leq  2\sqrt{\underline{a}},
\end{equation}
where $C(t;\omega)$ is given by \eqref{minimal-wave-speed-eq1}.

\smallskip

 Indeed, fix $\omega\in\Omega_0$,  $0<\mu<\sqrt{\underline{a}}$, and  $s\in\R$.  Note that the function $(x,t)\mapsto \tilde{U}(t,x;\omega):= \frac{1}{2} U(x-C(t;\omega),\theta_{t}\omega) $ satisfies
\begin{equation*}
\partial_t\tilde{U}(t,x;\omega)\leq \partial_{xx}\tilde{U}(t,x;\omega)+a(\theta_t\omega)(1-\tilde{U}(t,x;\omega))\tilde{U}(t,x;\omega),\quad t,x\in\R,
\end{equation*}
 and that, by \eqref{g-e2-2-1},
\begin{equation*}
\tilde{U}(s,x;\omega)=\frac{1}{2} U(x-C(s;\omega),\theta_{s}\omega) \leq U^{\mu}(x-C(s;\omega)+{ x_{\mu}(\theta_s\omega)},\theta_{s}\omega)\quad \forall x,s\in\R.
\end{equation*}
Hence, by comparison principle for parabolic equations,
\begin{align*}
\tilde{U}(t,x;\omega)\leq& u(t-s,x;U^\mu(\cdot-C(s;\omega)+x_\mu(\theta_s\omega),\theta_s\omega),\theta_s\omega)\cr
=&U^\mu(x-C(s;\omega)+x_\mu(\theta_s\omega)-C(t-s;\theta_{t}\omega,\mu),\theta_{t}\omega), \quad \forall\ x\in\R, \, t\geq s.
\end{align*}
In particular, taking $x=C(t,\omega)$, we obtain
$$
\frac{1}{4}\leq U^\mu(C(t;\omega)-C(s;\omega)+x_\mu(\theta_s\omega)-C(t-s;\theta_{t}\omega,\mu),\theta_{t}\omega), \quad \forall\ x\in\R, \, t\geq s.
$$
Which implies that there is $N_{\mu}(\omega)$ such that
\begin{equation}\label{upper-bound-eq-for-ave-mean}
C(t;\omega)-C(s;\omega)\leq C(t-s;\theta_{t}\omega,\mu),\theta_{t}\omega)+  { x_{\mu}(\theta_s\omega)}+ N_{\mu}(\omega),\quad \forall t>s.
\end{equation}
This together with \eqref{g-e2} and the fact that $C(t-s;\theta_{t}\omega,\mu),\theta_{t}\omega)=C(t;\omega,\mu)-C(s;\omega,\mu)$ (see \eqref{average-speed-eq1}) yield,
$$
\liminf_{t-s\to\infty}\frac{C(t;\omega)-C(s;\omega)}{t-s}\leq \liminf_{t-s\to\infty}\frac{C(t;\omega,\mu)-C(s;\omega,\mu)}{t-s}=\frac{\mu^2+\sqrt{\underline{a}}}{\mu},\quad \forall\ 0<\mu<\sqrt{\underline{a}}.
$$
Which completes the proof of \eqref{g-e2-2-3}.

\medskip

\noindent{\bf Step 5}. The proof of $$
\liminf_{t-s\to\infty}\frac{C(t;\omega)-C(s;\omega)}{t-s}\geq 2\sqrt{\underline{a}}.
$$
follows from Theorem \ref{Existence of random transition front} (3), whose proof will be given in the next subsection.
\end{proof}

\subsection{Proof of Theorem \ref{Existence of random transition front}(3)}

In this subsection, we present the proof of Theorem \ref{Existence of random transition front}(3).

\begin{proof}[Proof of Theorem \ref{Existence of random transition front}(3)] Suppose that $u(t,x;\omega)=U(x-C(t;\omega);\omega)$ is a random transition wave of \eqref{Main-eq} connecting $u\equiv 0$ and  $u\equiv 1$ satisfying \eqref{tail-eq} and $0<\varepsilon\ll 1$.
$$ \inf_{x\le z}\inf_{s\in\R}U(x,\theta_s\omega)>0, \quad \forall\ z\in\R$$
Thus, we can chose $u_0(\cdot,\omega)\in X^+_c$ such that
$$
u_0(x,\omega)\leq U(x,\theta_s\omega),\quad \forall \ s\in\R.
$$
Hence, it follows from Theorem \ref{spreading-speeds-tm} and comparison principle for parabolic equations,
\begin{align*}
1=&\lim_{t\to\infty}\inf_{s\in\R}u(t,(2\sqrt{\underline{a}}-\varepsilon)t;u_0(\cdot;\omega),\theta_s\omega)\cr
\leq& \liminf_{t\to\infty}\inf_{s\in\R}u(t,(2\sqrt{\underline{a}}-\varepsilon)t;U(\cdot,\theta_s\omega),\theta_s\omega)\cr
=&\liminf_{t\to\infty}\inf_{s\in\R}U((2\sqrt{\underline{a}}-\varepsilon)t-C(t;\theta_{s}\omega),\theta_{s+t}\omega)
\end{align*}
This together with the fact that $C(t+s;\omega)-C(s;\omega)=C(t;\theta_s\omega)$ (see \eqref{average-speed-eq1}) yield that there is a constant $K(\omega)$ such that
$$
(2\sqrt{\underline{a}}-\varepsilon)t\leq C(t+s;\omega)-C(s;\omega)+K(\omega),\quad  \forall \ t>0, \ s\in\R,
$$
Thus,
$$  2\sqrt{\underline{a}}-\varepsilon\le \liminf_{t\to\infty}\inf_{s\in\R}\frac{C(t+s;\omega)-C(s;\omega)}{t}= \underline{c}.$$
Letting $\varepsilon\to0$, the result follows.
\end{proof}

\section{Stability of random transition fronts}

In this section, we study the stability of random transition fronts of \eqref{Main-eq} established in the previous section.
We also study the existence and stability of random transition fronts of \eqref{real-noise-eq}.

\subsection{Stability of random transition fronts of \eqref{Main-eq}}

In this subsection, we study the stability of random transition fronts of \eqref{Main-eq}.  The following is the main theorem of this section.

\begin{tm}\label{Stability tm}
Assume that {\bf (H1)} hold. Then given $\mu\in(0, \underline{\mu}^*)$, the random wave solution $u(t,x)=U(x-C(t;\omega,\mu),\theta_t\omega)$ with $\lim_{x\to\infty}\frac{U(x;\theta_{t}\omega)}{e^{-\mu x}}=1$ and $C(t;\omega,\mu)=\int_0^tc(s;\omega,\mu)ds$  ($c(t;\omega,\mu)=\frac{\mu^2+a(\theta_t\omega)}{\mu}$) is asymptotically stable, that is, for any $\omega\in\Omega_0$ and
 $u_0\in C^b_{\rm unif}(\R)$ satisfying that
\begin{equation}\label{Stability - eq1 }
\inf_{x\leq x_0}u_0(x)>0 \quad \forall\, x_0\in\R, \quad \lim_{x\to\infty}\frac{u_0(x)}{U(x-C(0;\omega,\mu),\omega)}=1,
\end{equation}
there holds
$$
\lim_{t\to\infty}\Big\| \frac{u(t,\cdot;u_0,\omega)}{U(\cdot-C(t;\omega,\mu),\theta_t\omega)}-1 \Big\|_{\infty}=0.
$$
\end{tm}

To prove the above theorem, we first prove some lemma.

\begin{lem}\label{stabolity-lm2} Let $u_0\in C^b_{\rm unif}(\R)$ satisfy \eqref{Stability - eq1 }. Then for any $\omega\in\Omega_0$,  there holds
\begin{equation}
 \lim_{x\to\infty}\frac{u(t,x+C(t;\omega,\mu);u_0,\omega)}{e^{-\mu x}}=1 \quad \text{uniformly in } t\geq 0.
\end{equation}
where $C(t;\omega,\mu)=\int_0^ t c(s;\omega,\mu)ds$ ($c(t;\omega,\mu)=\frac{\mu^2+a(\theta_t\omega)}{\mu}$) and $\mu$ is given by Theorem \ref{Stability tm}.
\end{lem}

\begin{proof} Since $u_0$ satisfies \eqref{Stability - eq1 }, then for every $\varepsilon >0$, there is $x_{\varepsilon;\omega}\gg 1$ such that
$$
1-\varepsilon\leq \frac{u_0(x+C(0;\omega,\mu))}{U(x,\omega)}\leq 1+\varepsilon\quad \forall\, x\geq x_{\varepsilon;\omega}.
$$
Let $A_\omega(t)$ be as in Lemma \ref{sub-sol-lem}.
Since $e^{-\mu x}-d_{\omega}e^{A_\omega(t)-\tilde{\mu}x}\leq U(x,t)\leq e^{-\mu x}$, then
\begin{equation}\label{stab-eq-010}
(1-\varepsilon)e^{-\mu x}-(1-\varepsilon)d_{\omega}e^{A_\omega(0)-\tilde{\mu}x}\leq u_0(x+C(0;\omega,\mu))\leq (1+\varepsilon)e^{-\mu x},\quad \forall \, x\geq x_{\varepsilon;\omega}.
\end{equation}
We claim that there is $d\gg 1$ such that
\begin{equation}\label{claim1-stab-lem2}
(1-\varepsilon)e^{-\mu x}-de^{A_\omega(0)-\tilde{\mu}x} \leq u_0(x+C(0;\omega,\mu))\leq (1+\varepsilon)e^{-\mu x}+de^{A_{\omega}(0)-\tilde{\mu}x} \quad\, \forall\, x\in\R.
\end{equation}
Indeed, observe that
$$
\|u_0\|_{\infty}e^{\tilde{\mu}x_{\varepsilon;\omega}+|A_{\omega}(0)|}e^{A_\omega(0)-\tilde\mu x}\geq\|u_0\|_{\infty}e^{\tilde{\mu}x_{\varepsilon,\omega}}e^{-\tilde\mu x_{\varepsilon;\omega}}\geq u_0(x+C(0;\omega,\mu)), \quad \forall\, x\leq x_{\varepsilon,\omega}.
$$
Hence
\begin{equation}\label{stab-eq-011} u_0(x+C(0;\omega,\mu))\leq (1+\varepsilon)e^{-\mu x}+d_{\varepsilon,\omega
}e^{A_\omega(0)-\tilde{\mu}x} \quad\forall\, x\in\R,
\end{equation}
where $d_{\varepsilon;\omega}=:\|u_0\|_{\infty}e^{\tilde{\mu}x_{\varepsilon;\omega}+|A_\omega(0)|}$. On the other hand, for every $d>1$, the function $\R\ni x\mapsto (1-\varepsilon)e^{-\mu x}-d e^{A_\omega(0)-\tilde{\mu}x}$ attains it maximum value at $x_{d}=\frac{\ln(\frac{d\tilde{\mu}e^{A_\omega(0)}}{(1-\varepsilon)\mu})}{\tilde{\mu}-\mu}$. Hence, using the fact that $\lim_{d\to\infty}x_{d}=\infty$ and $$
\lim_{d\to\infty}((1-\varepsilon)e^{-\mu x_d}-d e^{A_\omega(0)-\tilde{\mu}x_d})=0,
$$
there is $\tilde{d}_{\varepsilon;\omega}\gg (1-\varepsilon)d_{\omega}$ such that  $x_{\tilde{d}_{\varepsilon;\omega}}\geq x_{\varepsilon;\omega}$ and  $$
(1-\varepsilon)e^{-\mu x_{\tilde{d}_{\varepsilon;\omega}}}-\tilde{d}_{\varepsilon;\omega} e^{A_\omega(0)-\tilde{\mu}x_{\tilde{d}_{\varepsilon;\omega}}} \leq \inf_{x\leq x_{\varepsilon;\omega}}u_0(x+C(0;\omega,\mu)).$$
Combining this with \eqref{stab-eq-010}, we obtain that
\begin{equation}\label{stab-eq-012}
(1-\varepsilon)e^{-\mu x}-d e^{A_\omega(0)-\tilde{\mu}x} \leq u_0(x+C(0;\omega,\mu))\quad \forall\, x\in\R,\,\, \forall\, d\geq  \tilde{d}_{\varepsilon;\omega}.
\end{equation}
Therefore it follows from \eqref{stab-eq-011} and \eqref{stab-eq-012} that  \eqref{claim1-stab-lem2} holds for every $d\geq \max\{\tilde{d}_{\varepsilon;\omega} , d_{\varepsilon;\omega} \}$.
By direct computation as in the proof of Lemma \ref{sub-sol-lem}, it holds that for $d\gg 1$,
$$\mathcal{G}^{\omega,\mu}((1-\varepsilon)e^{-\mu x}-d e^{A_\omega(t)-\tilde{\mu}x}) \leq 0 , \quad \text{a.e in } t $$
on the set $D_{\varepsilon}:=\{(x,t)\in\R\times\R^+ \ |\ (1-\varepsilon)e^{-\mu x}-d e^{A_\omega(t)-\tilde{\mu}x}\geq 0 \}$. Thus, since $u(t,x+C(t;\omega,\mu);u_0,\omega)\ge 0$, comparison principle for parabolic equations yields that
\begin{equation} \label{stab-eq-013}
(1-\varepsilon)e^{-\mu x}-d e^{A_\omega(t)-\tilde{\mu}x}\leq u(t,x+C(t;\omega,\mu);u_0,\omega) \quad \forall\, x\in\R,\,\,\forall\, t\ge 0, \,\, d\gg 1.
\end{equation}
Similarly, it holds that
$$
\mathcal{G}^{\omega,\mu}((1+\varepsilon)e^{-\mu x}+d e^{A_\omega(t)-\tilde{\mu}x})\geq 0, \quad x\in\R,\,\, t\in\R.
$$
Thus, by comparison principle for parabolic equations, it holds that
$$
u(t,x+C(t;\omega,\mu);u_0,\omega)\leq (1+\varepsilon)e^{-\mu x}+d e^{A_\omega(t)-\tilde{\mu}x} \quad \forall\, x\in\R,\,\,\forall\, t\ge0, \,\, d\gg 1.
$$
Since $\varepsilon>0$ is arbitrarily chosen, the last inequality combined with \eqref{stab-eq-013} yield that
$$\lim_{x\to\infty}\frac{u(t,x+C(t;\omega,\mu);u_0,\omega)}{e^{-\mu x}}=1 \quad \text{uniformly in } t\ge 0. $$
So, the lemma follows.
\end{proof}

\begin{proof}[Proof of Theorem \ref{Stability tm}] Fix $\omega\in\Omega_0$. Let $u_0\in C^b_{\rm unif}(\R)$ satisfying \eqref{Stability - eq1 }. Then there is $\alpha\geq 1$ such that
$$ \frac{1}{\alpha}\leq \frac{u_0(x)}{U(x-C(0;\omega,\mu),\omega)}\leq \alpha,\quad \forall\, x\in\R.$$
Hence, comparison principle for parabolic equations implies that
$$
u(t,x;u_0,\omega)\leq u(t,x;\alpha U(\cdot-C(0;\omega,\mu),\omega)), \quad\ \forall\, x\in\R,\,\, \forall\, t\geq 0
$$
and
$$
U(x-C(t;\omega,\mu),\theta_t\omega) \leq u(t,x;\alpha u_0 ,\omega), \quad\ \forall\, x\in\R,\,\, \forall\, t\geq 0.
$$
Note that $$
(\alpha u)_t \,  \le\,  (\alpha u)_{xx} + a(\theta_t\omega)(\alpha u)(1-\alpha u).
$$
Hence it follows from comparison principle for parabolic equations that
$$
U(x-C(t;\omega,\mu),\theta_{t}\omega)\leq \alpha u(t,x;u_0,\omega), \quad \forall t\geq 0.
$$
Similarly, we have that
$$
u(t,x;u_0,\omega) \leq \alpha U(x-C(t;\omega,\mu),\theta_{t}\omega), \quad \forall t\geq 0.
$$
Thus for every $t\ge 0$, there is a unique $\alpha(t)\geq 1$ satisfying
\begin{equation}\label{part-metric}
\alpha(t):=\inf\{  \alpha\geq 1\ \ |\  \ \frac{1}{\alpha}\leq \frac{u(t,x;u_0)}{U(x-C(t;\omega,\mu),\theta_t\omega)}\leq \alpha \ \text{for every } x\in\R \}.
\end{equation}
Furthermore,  we have that $\alpha(t)\leq \alpha(\tau)$ for every $0\leq \tau\leq t$. Therefore
$$
\alpha_{\infty}:=\inf\{\alpha(t) \ |\ t\geq 0 \}=\lim_{t\to\infty}\alpha(t).
$$
Note that, to complete the proof of Theorem \ref{Stability tm}, it is enough to show that $\alpha_{\infty}=1$.

It is clear that $\alpha_{\infty}\geq 1$. Suppose by contradiction that $\alpha_{\infty}>1$.  Let $1<\alpha<\alpha_{\infty}$ be fixed. It follows from Lemma \ref{stabolity-lm2} that there is $x_{\alpha}\gg 1$ such that
\begin{equation}\label{s-eq01}
\frac{1}{\alpha}\leq \frac{u(t,x+C(t;\omega,\mu);u_0,\omega)}{U( x,\theta_t\omega)}\leq \alpha, \quad \forall\, x\geq x_{\alpha}, \,\, \forall\, t\ge 0.
\end{equation}
Let us set
$$
m_{\alpha}:=\frac{1}{\alpha_0}\inf_{t\geq 0, \, \, x\leq x_{\alpha}} U(x;\theta_t\omega )(>0),
$$
where $\alpha_0=\alpha(0)=\sup_{t\geq 0}\alpha(t)$. Hence it follows from the definition of $\alpha_0$ that
$$
m_{\alpha}\leq \min\{u(t,x+C(t;\omega,\mu);u_0,\omega), U(x,\theta_{t}\omega)\},\quad  \forall\, x\leq x_{\alpha},\,\,\forall\, t\geq 0.
$$



 Now, since {\bf (H1)} holds, there is $T=T(\omega)\ge 1$ such that
\begin{equation}\label{w-00101}
0<\frac{\underline{a}T}{2}< \int_{s}^{s+T}a(\theta_\tau\omega)ds<2\overline{a} T<\infty, \quad \forall s\in\R.
\end{equation}
Let $0<\delta\ll 1$ satisfy
\begin{equation}\label{w-00100}
\alpha< e^{-2\delta T\overline{a}}\alpha_\infty \quad \text{and}\quad \left( (\alpha_\infty-1)-\alpha_0(1-e^{-2\delta T\overline{a}}) \right)m_\alpha>\delta.
\end{equation}

We claim that
\begin{equation}\label{s-eq00060}
\alpha((k+1)T)\leq e^{-\delta\int^{(k+1)T}_{kT}a(\theta_s\omega)ds}\alpha( kT), \quad \forall \, k\geq 0.
\end{equation}
Indeed, by taking $u_{k}(t,x)=e^{\delta\int_{kT}^{t+kT}a(\theta_s\omega)ds}u(t+kT,x+C(t+kT;\omega,\mu);u_0,\omega)$,
$U_{k}(t,x)=U(x;\theta_{t+kT}\omega)$,  $a_{k}(t)=a(\theta_{t+kT}\omega)$, and $\alpha_k=\alpha(kT)$ it follows from \eqref{w-00101} that
\begin{align}\label{s-eq060}
\partial_t u_{k}&=  \delta a_k(t) u_{k} +\partial_{xx}u_{k}+ \frac{\mu^2+a_{k}(t)}{\mu}\partial_x u_{k} +a_{k}(t)\Big(1-u(t+kT,x+C(t+kT;\omega,\mu);u_0,\omega)\Big)u_{k}\nonumber\\
 &= \partial_{xx}u_{k}+ \frac{\mu^2+a_{k}(t)}{\mu}\partial_x u_{k} +a_{k}(t)(1-u_{k})u_{k}  + a_{k}(t)\Big( (1-e^{-\delta \int_{kT}^{t+kT}a(\theta_s\omega)ds})u_{k}+\delta \Big)u_{k}\nonumber\\
 & \leq  \partial_{xx}u_{k}+ \frac{\mu^2+a_{k}(t)}{\mu}\partial_x u_{k} +a_{k}(t)(1-u_{k})u_{k}  + a_{k}(t)\Big( (1-e^{-2\delta T \overline{a}})u_{k}+\delta \Big)u_{k}
\end{align}
for every $t\in(0,T)$, $x\in\R$, and $k\geq 0$. On the other hand, it follows from \eqref{w-00100} and the fact that $\alpha_\infty\leq \alpha_k\leq \alpha_0$, that
\begin{align}\label{s-eq070}
 &\partial_t(\alpha_{k} U_{k})-\partial_{xx}(\alpha_{k} U_{k})-\frac{\mu^2+a_{k}(t)}{\mu}\partial_x(\alpha_{k} U_{k}) \cr
 =&  a_{k}(t)(1-U_{k})(\alpha_{k} U_{k})\cr
 =& a_{k}(t) (1-(\alpha_k U_k))(\alpha_k U_k) +a_k(t)\left( (1-e^{-2\delta T\overline{a}})(\alpha_k U_k) +\delta \right)(\alpha_k U_k)  \cr
 &  + a_k(t)\left( \left( (\alpha_k-1)-(1-e^{-2\delta T\overline{a}})\alpha_k\right) U_{k} -\delta \right)(\alpha_k U_k)  \cr
 \ge &  a_{k}(t) (1-(\alpha_k U_k))(\alpha_k U_k) +a_k(t)\left( (1-e^{-2\delta T\overline{a}})(\alpha_k U_k) +\delta \right)(\alpha_k U_k)  \cr
 &  + a_k(t)\left( \left( (\alpha_\infty-1)-(1-e^{-2\delta T\overline{a}})\alpha_0\right)m_\alpha -\delta \right)(\alpha_k U_k)  \cr
 \ge &  a_{k}(t) (1-(\alpha_k U_k))(\alpha_k U_k) +a_k(t)\left( (1-e^{-2\delta T\overline{a}})(\alpha_k U_k) +\delta \right)(\alpha_k U_k)
\end{align}
for $  x\leq x_{\alpha}$, $0\leq t\leq T$, and $k\geq 0$. Therefore, it follows from the definition of $\alpha_{k}$, \eqref{s-eq01}, the fact that $e^{\delta\int_{kT}^{(k+1)T}a(\theta_{s}\omega)ds} \alpha\leq \alpha_{\infty}\leq \alpha_{k}$, and comparison principle for parabolic equations that
\begin{equation*}
e^{\delta\int_{kT}^{t+kT}a(\theta_{s}\omega)ds}u(t+kT,x+C(t+kT;u_0,\omega)\leq \alpha_{k} U(x,\theta_{t+kT}\omega), \quad \forall\, x\leq x_{\alpha},\, t\in[0,T],\  k\geq 0.
\end{equation*}
That is
$$
u(t+kT,x+C(t+kT;\omega);u_0)\leq e^{-\delta\int_{kT}^{(k+1)T}a(\theta_{s}\omega)ds}\alpha_{k} U(x,\theta_{t+kT}\omega), \quad \forall\, x\leq x_{\alpha},\, t\in[0,T], \ k\ge 0.
$$
Since, $\alpha \leq e^{-\delta\int_{kT}^{(k+1)T}a(\theta_{s}\omega)ds}\alpha_{\infty} \leq e^{-\delta\int_{kT}^{(k+1)T}a(\theta_{s}\omega)ds}\alpha_{k}$, it follows from \eqref{s-eq01} that
$$
u(t+kT,x+C(t+kT;\omega);u_0,\omega)\leq e^{-\delta\int_{kT}^{t+kT}a(\theta_{s}\omega)ds}\alpha_{k} U(x,\theta_{t+t_{0,k}^{\varepsilon}}\omega), \,\, \forall\, x\geq x_{\alpha},\, t\in[0,T],\ k\geq 0.
$$
Therefore, for every $k\geq 1$, it holds that
\begin{equation}\label{s-eq06}
u(t+kT,x+C(t+kT;u_0,\omega);u_0,\omega)\leq e^{-\delta\int_{kT}^{t+kT}a(\theta_{s}\omega)ds}\alpha_{k} U(x,\theta_{t+kT}\omega), \,\, \forall\, x\in\R,\, t\in[0, T].
\end{equation}

Similarly, interchanging $u_{k}$ and $U_{k}$ in  \eqref{s-eq060} and \eqref{s-eq070}, we obtain that
\begin{equation}\label{s-eq07}
U(x,\theta_{t+kT}\omega)\leq e^{-\delta\int_{kT}^{t+kT}a(\theta_{s}\omega)ds}\alpha_k u(t+kT,x+C(t+kT;u_0,\omega);u_0,\omega) , \quad \forall\, x\in\R,\, t\in[0,T].
\end{equation}
Hence inequality \eqref{s-eq00060}  follows from \eqref{s-eq06} and \eqref{s-eq07}.  Thus, by induction we obtain that
$$
\alpha_{\infty}\leq \alpha((k+1)T)\leq e^{-\delta\sum_{i=0}^{k}\int_{iT}^{(i+1)T}a(\theta_s\omega)ds}\alpha(0)= e^{-\delta\int_0^{(k+1)T}a(\theta_s\omega)ds}\alpha_0, \quad \forall k\geq 0.
$$
But for $\omega\in\Omega_0$, it holds that $\int_0^{\infty}a(\theta_s\omega)ds=\infty$. Therefore, letting $k\to\infty$ in the last inequality, we obtain that $\alpha_{\infty}\leq 0$. Which is not possible because $\alpha_{\infty}\geq 1$. Therefore $\alpha_{\infty}=1$, which completes the proof of the Theorem.
\end{proof}

\subsection{Existence and stability of random transition fronts of \eqref{real-noise-eq}}

In this subsection, we consider the existence and stability of random transition fronts of
\eqref{real-noise-eq} by applying the established results for  \eqref{Main-eq}. We have

 \begin{coro}
\label{real-noise-cor}
Assume {\bf (H3)}.
  Let $Y(\omega)$ be the random equilibrium solution of \eqref{real-noise-eq} given in \eqref{random-equilibrium-1}. Then for any given $0<\mu<1$, \eqref{real-noise-eq} has a random transition wave solution $u(x,t)=U(x-C(t;\omega,\mu);\theta_t\omega)$ with $C(t;\omega,\mu)=\int_0^t \frac{\mu^2+Y(\theta_s\omega)}{\mu^2}ds$ and connecting $u\equiv 0$ and the random equilibrium solution $Y(\omega)$. Moreover the following holds.
 \begin{itemize}
 \item[(i)] $0<U(x,\omega)<Y(\omega)$ for a.e. $\omega\in\Omega$.

 \item[(ii)] $\lim_{x\to-\infty}\frac{U(x,\theta_t\omega)}{Y(\theta_t\omega)}=1$ and $\lim_{x\to\infty} \frac{U(x,\theta_t\omega)}{e^{-\mu x}Y(\theta_t\omega)}=0$ uniformly in $t$ for a.e. $\omega\in\Omega$.

 \item[(iii)] For a.e. $\omega\in\Omega$ and every $u_0\in C^b_{\rm unif}(\R)$ satisfying that
  \begin{equation}\label{real-noise-thm-eq}
\inf_{x\leq x_0}u_0(x)>0 \quad \forall\ x_0\in\R,\quad \lim_{x\to\infty}\frac{u_0(x)}{U(x-C(0;\omega,\mu),\omega)}=1,
  \end{equation}
  we have
  $$
\lim_{t\to\infty}\Big\|\frac{u(t,\cdot;u_0,\omega)}{U(\cdot-C(t;\omega,\mu),\omega)}-1\Big\|_{\infty}  =0.
  $$
 \end{itemize}
\end{coro}

\begin{proof}
 First of all, observe that $0<Y(\omega)<\infty$. Let $\tilde u=\frac{u}{Y(\theta_t\omega)}$ and drop the tilde. We have
\begin{equation}
\label{real-noise-eq1}
u_t=u_{xx}+Y(\theta_t\omega) u(1-u).
\end{equation}
Clearly, \eqref{real-noise-eq1} is of the form \eqref{Main-eq} with $a(\omega)=Y(\omega)$.

Let $0<\mu<1$ be given. Note that by Theorems  \ref{real-noise-tm1},  \ref{Existence of random transition front}, and  \ref{Stability tm}, \eqref{real-noise-eq1} has a random transition wave solution $0<\tilde u(t,x)<1$, $\tilde{u}(t,x)=\tilde{U}(x-C(t;\omega,\mu),\theta_t\omega)$ connecting $\tilde{u}\equiv 0$ and $\tilde{u}\equiv 1$ and   satisfying
$$
\lim_{x\to-\infty}\tilde{U}(x,\theta_t\omega)=1 \quad \text{and}\quad \lim_{x\to\infty}\frac{\tilde U(x,\theta_t\omega)}{e^{-\mu x}} =1, \quad \text{uniformly in }\ t,$$
and
$$
\lim_{t\to\infty}\Big\| \frac{\tilde u(t,\cdot;\tilde u_0,\omega)}{\tilde U(\cdot-C(t,\omega),\theta_t\omega)}-1 \Big\|_{\infty}=0
$$
for every $u_0\in C^b_{\rm unif}(\R)$ satisfying \eqref{real-noise-thm-eq}.
Let us set $U(x,\omega)=\tilde{U}(x,\omega)Y(\omega)$ for $x\in\R, \ \omega\in\Omega$. Thus, the function $u(t,x)=U(x-C(t;\omega,\mu),\theta_t\omega)$ is a stable random transition wave solution of \eqref{real-noise-eq} connecting $0$ and $Y(\theta_t\omega)$ and satisfies the desired properties.
\end{proof}

\section{Transition front of nonautonomous Fisher-KPP equations}\label{Sec nonautonomous}

In this section we consider the nonautonomous Fisher-KPP equation \eqref{nonautonomous-eq}.

For given $u_0\in C_{\rm unif}^b(\R)$ with $u_0\ge 0$, let $u(t,x;u_0,\sigma_\tau a_0)$ be the solution of
$$
u_t=u_{xx}+\sigma_\tau a_0(t) u(1-u),\quad x\in\R,\, t>0,
$$
with $u(0,x;u_0,\sigma_\tau a_0)=u_0(x)$, where $\sigma_s a_0(t)=a_0(s+t)$.
We have the following theorems on the existence and stability of transition front solutions, and spreading speeds  of \eqref{nonautonomous-eq}

 \begin{tm}
 \label{nonautonomous-thm1}
 Suppose that {\bf (H2)} holds.
  \begin{itemize}
 \item[(1)] For any given $\mu\in(0, \sqrt{\underline{a}_0})$,
      \eqref{nonautonomous-eq} has a transition wave solution $u(t,x)=U^\mu(x-C(t;\mu),t)$ with  $C(t;\mu)=\int_0^t c(s;\mu)ds$, where
\begin{equation}\label{Exist-2-tm-eq1}
c(t;\mu)=\frac{\mu^2+a_0(t)}{\mu},
 \end{equation}
and hence  $\underline{c}=\frac{\mu^2+\underbar a_0}{\mu} >\underline{c}^*_0$.
Moreover,
\begin{equation}\label{Exist-2-tm-eq2}
\lim_{x\to \infty}\sup_{t\in\R}\Big|\frac{U^\mu(x,t)}{e^{-\mu x}}-1\Big|=0 \quad \text{and}\quad \lim_{x\to -\infty}\sup_{t\in\R}|U^\mu(x,t)-1|=0.
\end{equation}

\item[(2)] There is a transition front of \eqref{nonautonomous-eq}  with least mean speed  $\underline{c}^*_0$.

\item[(3)] There is no transition front of \eqref{nonautonomous-eq}  with  least mean  speed less than $\underline{c}^*_0$.
\end{itemize}

 \end{tm}

\begin{tm}\label{nonautonomous-thm2}
Assume that {\bf (H2)} holds. Then given $\mu\in(0, \underline{\mu}^*)$, the transition wave solution $u(t,x)=U(x-C(t;\mu),t)$ with $\lim_{x\to\infty}\frac{U(x;t)}{e^{-\mu x}}=1$ and $C(t;\mu)=\int_0^tc(s;\mu)ds$  ($c(t;\mu)=\frac{\mu^2+a_0(t)}{\mu}$) is asymptotically stable, that is, for any $u_0\in C^b_{\rm unif}(\R)$ satisfying that
\begin{equation}\label{Stability-2-eq1}
\inf_{x\leq x_0}u_0(x)>0 \quad \forall\, x_0\in\R, \quad \lim_{x\to\infty}\frac{u_0(x)}{U(x-C(0;\mu),0)}=1,
\end{equation}
there holds
$$
\lim_{t\to\infty}\Big\| \frac{u(t,\cdot;u_0,a_0)}{U(\cdot-C(t;\mu),t)}-1 \Big\|_{\infty}=0.
$$
\end{tm}

Theorem  \ref{nonautonomous-thm1} (resp. Theorem \ref{nonautonomous-thm2})  can be proved by the similar arguments as those in Theorem \ref{Existence of random transition front} (resp. Theorem \ref{Stability tm}). In the following, we provide some indications for the proofs of Theorems \ref{nonautonomous-thm1} and \ref{nonautonomous-thm2}.

\begin{proof}[Indication of the proof of Theorem \ref{nonautonomous-thm1}]
For given $\mu>0$, let   $C(t;\mu)=\int_0^t c(s;\mu)ds$, where $c(t;\mu)$ is as in \eqref{Exist-2-tm-eq1}. Let
$u(t,x)=v(x-C(t;\mu),t)$. Then  $v(t,x)$ satisfies
\begin{equation}
\label{real-noise-moving-eq}
u_t=u_{xx}+c(t;\mu)v_x+a_0(t) v(1-v),\quad x\in\R.
\end{equation}

(1) It suffices to prove that for given $\mu\in(0,\sqrt{\underline{a}_0})$, \eqref{real-noise-moving-eq} has a solution
$v=U^\mu(t,x)$ satisfying \eqref{Exist-2-tm-eq2}.
To this end, for given $u_0\in C_{\rm unif}^b(\R)$ with $u_0\ge 0$, let
$v(t,x;u_0,\sigma_s a_0)$  be the solution of
\begin{equation}
\label{real-noise-moving-eq1}
u_t=u_{xx}+c(t;\mu)v_x+\sigma_s a_0(t) v(1-v),\quad x\in\R
\end{equation}
with $v(0,x;u_0,\sigma_s a_0)=u_0(x)$. Let $\phi_+^\mu(x)$ be defined as in \eqref{super-sol}, that is,
$$
\phi_+^\mu(x)=\min\{1,e^{-\mu x}\}\quad \forall\,\, x\in\R.
$$
Let $v^n(t,x)=v(t+n,x;\phi_+^\mu,\sigma_{-n}a_0)$. By the similar arguments as those in Theorem \ref{Existence of random transition front}(1),
$\lim_{n\to\infty}v^n(t,x)$ exists for all $t,x\in\R$ and $v=U^\mu(t,x):=\lim_{n\to\infty}v^n(t,x)$ is a solution of
 \eqref{real-noise-moving-eq} satisfying \eqref{Exist-2-tm-eq2}. This completes the proof of (1).

\smallskip

(2) Let $u_0^*$ be as in \eqref{u-0-def}, that is,
$$
u_0^*(x)=\begin{cases}
         1 , \quad x\leq 0\cr
         0 , \quad x>0.
       \end{cases}
$$
Then for any $t>0$, $u(t,x;u_0^*,\sigma_\tau a_0)$ is strictly decreasing in $x\in\R$ and
$u(t,-\infty;u_0^*,\sigma_\tau a_0)=1$, $u(t,\infty;u_0^*,\sigma_\tau a_0)=0$. Therefore, there is a unique $x(t;\tau)$ such that
$u(t,x(t;\tau);u_0^*,\sigma_\tau a_0)=1/2$. By the similar arguments as those in Theorem \ref{Existence of random transition front}(2).
it can be proved that $U(x,\tau):=\lim_{t\to\infty}u(t,x+x(t;\tau);u_0^*,\sigma_{-t+\tau}a_0)$ exists for all $x,\tau\in\R$, and
there is $C(t)$ such that $u(t,x)=U(x-C(t),t)$ is a transition front of \eqref{nonautonomous-eq} with least averge speed $c=\underline{c}^*_0$.

\smallskip
(3) Suppose that $u(t,x)=U(x-C(t),t)$ is a transition front solution of \eqref{Main-eq}. It can be proved by the similar arguments
as those in Theorem \ref{Existence of random transition front}(3) that $\underline{C}\ge \underline{c}_0^*$. (3) thus follows.
\end{proof}

\begin{proof}[Indication of the proof of Theorem \ref{nonautonomous-thm2}]
Suppose that $u_0\in C_{\rm unif}^b(\R)$ satisfies  \eqref{Stability-2-eq1}. By the similar arguments as those in  Theorem \ref{Stability tm}, it can be proved that
$$
\alpha(t)=\inf\{\alpha \ge 1\,| \frac{1}{\alpha}\le \frac{u(t,x;u_0,a_0)}{U(x-C(t;\mu),t)}\le\alpha\quad \forall\, x\in\R\}
$$
is well defined, and that
$\lim_{t\to\infty} \alpha(t)=1.$
The theorem then follows.
\end{proof}

 We conclude this section with some example of explicit function $a_0(t)$ satisfying {\bf (H2)}.

 \medskip

 Define the sequences $\{l_{n}\}_{n\geq 0}$ and $\{L_n\}_{n\geq 0}$ inductively by
 \begin{equation}\label{ln-Ln-def}
 l_0=0, \quad L_n=l_n+\frac{1}{2^{2(n+1)}}, \quad l_{n+1}=L_n+n+1, \ \ n\geq 0.
 \end{equation}
 Define $a_0(t)$ such that $a_0(-t)=a_0(t)$ for $t\in\R$ and
 \begin{equation}\label{a(t)-def}
 a_0(t)=\begin{cases}
 f_n(t)\qquad \text{if}\ t\in [l_n,L_n]\cr
 g_n(t) \qquad\ \text{if}\ t\in[L_n,l_{n+1}]\cr
 \end{cases}
 \end{equation}
 for $n\ge 0$,
 where $g_{2n}(t)=1$ and $g_{2n+1}(t)=2$ for $n\ge 0$, and $f_0(t)=1$, for $n\ge 1$, $f_{n}$ is H\"older's continuous on  $[l_{n},L_{n}]$, $f_n(l_n)=g_n(l_n)$, $f_{n}(L_{n})=g_n(L_{n})$, and satisfies
 $$
1\leq f_{2n}(t)\leq 2^n, \quad \max_{t}f_{2n}(t)=2^n, 
 $$
 and
 $$
\frac{1}{2^{n+1}}\leq f_{2n+1}(t)\leq 2, \quad \min_{t}f_{2n+1}(t)=2^{-(n+1)}. 
 $$
 It is clear  that $a_0(t)$ is locally H\"older's continuous, $
\inf a_{0}=0$, and  $\sup a_{0}=\infty$. Moreover, it can be verified that
$$
\underline{a_0}=1\quad \overline{a}_0=2.
$$
Hence $a_0(t)$ satisfies {\bf (H2)}.

\end{document}